\documentclass[]{amsart}

\usepackage{graphicx} 
\usepackage[a4paper,left=35mm,right=35mm,top=30mm,bottom=30mm,marginpar=25mm]{geometry}
\usepackage{amsmath}
\usepackage{amsthm}
\usepackage{amssymb}
\usepackage{tikz, pgf}
\usepackage{xcolor} 
\usepackage[displaymath,mathlines]{lineno} 
\usetikzlibrary{decorations.markings}
\usepackage[english]{babel}
\usepackage[T1]{fontenc}
\usepackage[utf8]{inputenc}
\usepackage{amsfonts}
\usepackage{mathrsfs}
\usepackage{xfrac}
\usepackage{lmodern}
\usepackage{fix-cm}
\usepackage{listings}
\usepackage{fancyvrb}
\usepackage{enumerate}   
\usetikzlibrary{arrows}
\usepackage{float}
\usepackage{subcaption}

\makeatletter  \renewcommand\@biblabel[1]{#1.}  \makeatother

\allowdisplaybreaks

\newcommand{\cal}{\mathcal}

\newcommand{\R}{{\mathbb{R}}}

\newcommand{\lra}{\longrightarrow}

\newcommand{\rad}{2 pt}
\newcommand{\colo}{black}

\newcommand{\ep}{\varepsilon}

\DeclareMathOperator{\N}{\mathbb{N}}

\DeclareMathOperator*{\argmin}{arg\,min}

\newcommand{\cB}{{\mathcal B}}

\newcommand{\cI}{{\mathcal I}}

\newcommand{\cN}{{\mathcal N}}

\newcommand{\cD}{{\mathcal D}}
\newcommand{\cC}{{\mathcal C}}

\newcommand{\cP}{{\mathcal P}}

\renewcommand{\phi}{\varphi}

\renewcommand{\d}{{\delta}}

\renewcommand{\l}{\lambda}

\newtheorem{theorem}{Theorem}
\numberwithin{theorem}{section}
\newtheorem{proposition}{Proposition}
\newtheorem{example}{Example}

\newtheorem{lemma}{Lemma}

\newtheorem{definition}{Definition}

\newtheorem{remark}{Remark}

\author{Fabio Bagagiolo}
\address{Dipartimento di Matematica \\
Universit\`a di Trento\\
Via Sommarive, 14, 38123 Povo (TN) Italy}
\email{fabio.bagagiolo@unitn.it }

\author{Luciano Marzufero}
\address{Dipartimento di Matematica \\
Universit\`a di Trento\\
Via Sommarive, 14, 38123 Povo (TN) Italy}
\email{luciano.marzufero@unitn.it }

\thanks{This work was partially supported by MUR-PRIN2020 Project (No. 2020JLWP23) ``Integrated mathematical approaches to socio-epidemiological dynamics''}
\title[Time-dependent mean-field game on networks]{\Large \bf A time-dependent switching mean-field game on networks motivated by optimal visiting problems}

\begin{document}

\maketitle
\date{}
\begin{abstract}
Motivated by an optimal visiting problem, we study a switching mean-field game on a network, where both a decisional and a switching time-variable are at disposal of the agents for what concerns, respectively, the instant to decide and the instant to perform the switch. Every switch between the nodes of the network represents a switch from $0$ to $1$ of one component of the string $p=(p_1,\ldots,p_n)$ which, in the optimal visiting interpretation, gives information on the visited targets, being the targets labeled by $i=1,\ldots,n$. The goal is to reach the final string $(1,\ldots,1)$ in the final time $T$, minimizing a switching cost also depending on the congestion on the nodes. We prove the existence of a suitable definition of an approximated $\varepsilon$-mean-field equilibrium and then address the passage to the limit when $\varepsilon$ goes to $0$.

\end{abstract}


\section{Introduction}
An optimal visiting problem in $\mathbb R^d$ is an optimal control problem where an agent has to visit (touch) a finite number of fixed targets (regions of $\mathbb R^d$) minimizing a suitable cost. The associated mean-field game problem may consist in considering a huge population of agents (even infinitely many) with the same goal and with the costs also depending on the congestion of the population. In order to write a Dynamic Programming Principle for the optimal visiting problem, some additional state-variables, taking into account which targets have been already visited or not, must be inserted. Such variables may be for example switching quantities as strings of $0$ and $1$, where $1$ in the $i$-position means that the target $i$ has been already visited and viceversa for $0$. Hence, starting from the string $p_o=(0, \ldots, 0)$, the goal can be seen as obtaining the string $\bar p=(1, \ldots, 1)$ paying as less as possible. Since all the possible strings $p$ are in a finite number and the switches must follow a hierarchical admissibility criterium, we interpret them as nodes of a direct network, where $p_o$ is the origin and $\bar p$ is the final destination. The problem can be seen then as the search for an optimal origin-destination path. Due to the dynamical feature of the optimal visiting problem in $\mathbb R^d$, in our network switching representation we keep the possibility for the agent to choose the sequence of instants to perform the switches, within a fixed time $T>0$. Again, the associated mean-field game consists in a huge population of agents where the choice of the optimal path is also affected by a congestion cost.
In particular, the agents want to touch or spend time on the nodes of a network, which represent the information on the visited targets, avoiding queues and congested spots. Inspired by the dynamical model in \cite{BFMP}, here we present a model in a pure switching form which, in some way, takes anyway into account a primitive structure of a continuous dynamics along the paths of a network (which is not present here). The main goal is to prove the existence of a mean-field equilibrium.

Our idea is then to study both the single-player problem and the crowd one without a real dynamics, i.e., without a controlled continuous trajectory for visiting the targets of the problem. For the single-player one, the state of the system is represented by a discrete variable $p$, which basically corresponds to the node of the network on which the agent is. Such a variable acts also as a switching discrete control at the agent disposal, that is, once the agent is on the node $p$, it has to choose optimally the next admissible subsequent node $p'$ after $p$. In this way, the agent switches to $p'$ and the state of the system becomes $p'$. In performing such a switch, the agent incurs a switching cost. A time-variable is accounted for the problem too. In particular, besides the switching discrete control variable $p$, the agent has to optimally choose the optimal time it is convenient to switch to the next node of the network. Moreover, all the admissible switches have to be performed within the fixed time $T$: if the agent reaches the final node before $T$, it pays an earliness penalization cost, while if it does not reach the final node and the time is over, it pays a time-loseness penalization cost.

In the mean-field case, we study the behavior of infinitely many players that have to solve the same single optimization problem as above, with the add of some kind of congestion cost dependence in the switching costs. After studying the single-player optimization problem and the properties of the corresponding value function, we face the problem of the existence of a mean-field equilibrium. This is done by performing a suitable fixed-point procedure for an approximated problem, and then we address the passage to the limit in the approximation. We need first an approximated problem because the switching mass-evolution, solution of the mean-field equilibrium problem, turns out to be piecewise continuous (even piecewise constant in some particular case) and this fact makes the standard compactness and convexity requirements for fixed-point results lacking in our case. Moreover, possibly due to non-uniqueness of the optimal control, we have to work with set-valued functions and, similarly as in \cite{BFMP}, \cite{bertucci2}, we must consider agents splitting into fractions, each one of them following one of the optimal behaviors. That suitably constructed and rather new approximation leads then to Theorem \ref{maintheorem}, which proves the existence of an approximated mean-field equilibrium via a fixed-point procedure for a suitable set-valued map. The passage to the limit in the approximation is then investigated in Theorem \ref{casoepsilon}, by assuming a suitable hypothesis on the optimal switching instants. Anyway, such a hypothesis can be satisfied by requiring some suitable conditions on the costs (see Remark \ref{esempiotempi}, the comments in \S\ref{generalcase} and Appendix \ref{appendixB}).

As regards the uniqueness of the equilibrium, usually, in the mean-field games theory, it is guaranteed by imposing a kind of monotonicity condition satisfied by the costs with respect to the mass of the agents (see \cite{lasrylions}). In several cases, the adaptation of that property to uniqueness results does not require too much work because the studied problem almost naturally fits that condition. Our problem, due to many of its aspects, does not provide instead an immediate evident way to adapt such a condition. However, inspired by \cite{lasrylions}, in Appendix \ref{appendix}, using a monotonicity-type property, we give some easy examples and calculations which seem to be promising for a future and deeper study of the uniqueness.

In general, as aforementioned, the study of single-player optimal visiting problems requires an hybrid control framework in order to recover a dynamic programming property and hence to derive an Hamilton-Jacobi equation. More precisely, it requires a special framework able to include a memory of the targets already visited. The need of that memory feature, associated with the optimal visiting, dynamic programming and Hamilton-Jacobi equations has been presented in \cite{bagfesmar1, bagfesmar2}, where additional discrete state-variables were introduced. The use of a switching/discontinuous/hybrid memory, as in the present paper, was also used for a one-dimensional optimal visiting problem on a network in \cite{BFMP}, which basically inspired our model.

The model for a crowd of indistinguishable players is taken from the framework of mean-field games \cite{laslio,huamalcai,Gomes,Carda}, while the adaptation of the same hybrid structure to networks has been only very recently attempted, as in \cite{BFMP, bagmagpes} and, more generally, in \cite{camilli1, camilli2}. For other studies in mean-field games on networks, see also the recent preprint \cite{Gomes1}. Some works which share the same ideas to treat the mean-field case in the presence of switches in the dynamics of the problem are \cite{bertucci1, bertucci2} and \cite{bagfesmar2}, where a mean-field optimal stopping problem, possibly with sinks and sources, is discussed, and \cite{festa2018mean}, where a hybrid mean-field game is presented to model a multi-lane traffic flux of vehicles. Moreover, other applications of a similar mean-field model can be found in \cite{mor}, where a continuous and a discrete set of switching labels are introduced to study the case of a leader-follower dynamics.
For what concerns applicative motivations of our model, as already said, we start by the overlying optimal visiting problem. For instance, such a problem complies with tourists' flow which has to visit several points of interest both in an heritage city and in a museum environment (see \cite{BFMP}, \cite{cri} and the references therein). In \cite{bovy} it is instead given an example of a situation where, in a crowded environment, people have to perform a sequence of different operations in different places, such as in big airports or train stations. Still concerning the optimal visiting problem, but related to a single-player one, in \cite{bagfesmar3} it is given an example of a framework which is used to solve a series of applied problems arising from the sport of orienteering races.

Finally, as we already said, we rewrite the optimal visiting problem as an origin-to-destination one on a network, and of course the possible applications and literature on this kind of problems are very huge. However, our model seems to have other interesting applicative aspects as described here below.
\par\medskip
\noindent {\it Other possible interpretations and applications of the model.} Besides the mainly motivating optimal visiting problem, another possible interpretation of our model is as a mean-field game for the so-called (single-stage) optimal job scheduling or for the similar open-shop scheduling problem in operations research (see for example \cite{Pinedo, GS}).
In this model, every agent represents a so-called job scheduler that has to produce its own optimal schedule. The machines, given as datum of the problem, are supposed identical and they can be interpreted as the targets of the visiting problem and then as the nodes of the network. The jobs are also given as a datum and they are the tasks that every job scheduler has to perform on each machine within the fixed time $T$. The optimal time and the optimal node chosen by the agents in the optimization process represent the processing time of a job (or of one or more operations) that has to be worked on a machine. Furthermore, if an agent reaches the final node (which means to have worked on all machines) before time $T$, then it has to pay a penalization cost: every job scheduler has to spend enough time on each machine to perform its job (or operations) and going faster may be penalizing. In the mean-field game formulation we may have a huge number of job-scheduler (agents) and hence, differently from the standard assumptions in the job scheduling, every machine has to be able to work more than one job (or more than one operation of a job) at a time. However, as usual, every job (or operation) can not be processed simultaneously at more than one machine. Then, the goal of every job scheduler is to optimize its schedule, minimizing a cost which, among others, penalizes queues and job-congestion on each machine. In some sense, the agents have to possibly use the ``most available'' machine. In the mean-field equilibrium situation, the job schedulers perform their optimal schedule: the best allocation of every job to the available machines together with the corresponding optimal processing times.

Still in the scheduling-like framework, we believe that another possible interpretation of the problem may be as an optimal co-flow scheduling, possibly with a deadline (see for example \cite{coflow}). In this model, several prescribed units of data (the demand) must be transferred from some sources to some sinks (nodes of a network) along some prescribed channels (edges of the networks) with fixed capacity. Each one of those transfers is a single flow. A co-flow is a set of a finite number of single flows and it has its own degree of priority. The optimization problem is to schedule all the single flows, without violating the capacity constraint, and minimizing the completion times of the co-flows, averaged by their priorities.

The paper is organized as follows. In Section \ref{optvisprob}, we introduce the time-dependent optimal switching problem, justified by an optimal visiting one, for a single agent and for a crowd, giving all the theoretical elements and hypotheses that motivate the use of a switching feature on a network. In Section \ref{fixedmass}, we study the well-position of such a problem with fixed mass, i.e, as a single-player optimization problem, showing the regularity of the value function and a dynamic programming property. In Section \ref{flowsofag}, we start the study for a population of agents by formally introducing the continuity equations for the flow and a suitable interpretation of a possible solution. Then, in Section \ref{apprmfg}, we introduce the mean-field game system of our problem, by proving at first the existence of an approximated $\ep$-mean-field equilibrium through a fixed-point procedure. Finally, in Section \ref{casolimite}, we address the passage to the limit as $\ep\to0$.

\section{The time-dependent optimal switching problem on the network}
\label{optvisprob}
Let $\{\cN_j\}_{j=1,\ldots,N}\subset\R^d$ be the collection of $N$ targets of the optimal visiting problem. As explained in the Introduction, we consider the set of the $N$-strings $p=(p^1, p^2, \ldots, p^N)\in\cI=\{0, 1\}^N$, which has cardinality equal to $2^N$, and which we interpret  as the nodes of our network. In particular, $p^i=1$ means that the $\cN_i$ has already been visited and viceversa for $p^i=0$.
The node $(1, 1, \ldots, 1)$ is the final destination and, once reached, the game ends.

By the meaning of the strings $p$, at every switch, just one component may change and it can do that only from $0$ to $1$. Such a component corresponds to the visited target. For example, for $N=4$ targets, if $p_1=(1, 0 ,0, 0)$, $p_2=(1, 1, 0, 0)$, $p_3=(0, 1, 1, 0)$ and $p_4=(1, 1, 1, 0)$, then from $p_1$ we can not switch to $p_3$ otherwise we lose the information that the first target has been already visited. Moreover we can not switch to $p_4$ directly since, as we said, at every switch just one component flips.

Hence, to any $p\in\cI$ we associate the number $k_p$ given by the sum of the components of $p$, that is $k_p=p^1+\ldots+p^N$. In other words, $k_p$ is the number of ``$1$'' in $p$, that is the number of the visited targets. Then, for any $p\in\cI$, we denote by $\cI_p$ the set of all possible new variables (nodes) in $\cI$ after a switch from $p$:
$$
\cI_p:=\{\tilde p\in\cI:\text{for every } i=1,\ldots,N, \ \tilde p^i=p^i+1\ \text{ if }\ p^i\neq1\ \text{ and }\ k_{\tilde p}=k_p+1\}.
$$
We observe that, in particular, $\cI_{\bar p}=\emptyset$, where $\bar p=(1, 1, \ldots,1)$.
\begin{example}
For $N=3$ targets, all the possible ways to visit them are $N! = 3! = 6$ as we can see in Figure \ref{figure1}. Hence our corresponding direct network is represented in Figure \ref{figure2}, where $p_o=(0, 0, 0)$ is the origin and $\bar p=(1, 1, 1)$ is the final destination. We then have for example $\cI_{p_o}=\{(1, 0, 0), (0, 1, 0), (0, 0, 1)\}$ and $\cI_{\tilde p=(0, 0, 1)}\{(1, 0, 1), (0, 1, 1)\}$.
\begin{figure}
\centering
\begin{subfigure}{0.25\textwidth}
\begin{tikzpicture}[very thick,decoration={
    markings,
    mark=at position 1 with {\arrow{>}}}
    ]
\filldraw[\colo] (0,3) circle (\rad) node [anchor=east] {$\cN_1$};
\filldraw[\colo] (1,4) circle (\rad) node [anchor=south] {$\cN_2$};
\filldraw[\colo] (2,3) circle (\rad) node [anchor=west] {$\cN_3$};
\node at (1,1.5) {\Huge $\ast$};

\draw[postaction={decorate}] (0.8,1.8) --(0.1,2.85);
\draw[postaction={decorate}] (0.1,3.1) --(0.8,3.8);
\draw[postaction={decorate}] (1.2,3.8) --(1.8,3.2);
\end{tikzpicture}
\caption{First way}
\end{subfigure}%
\hspace{7pt}
\begin{subfigure}{0.25\textwidth}
\begin{tikzpicture}[very thick,decoration={
    markings,
    mark=at position 1 with {\arrow{>}}}
    ]
\filldraw[\colo] (0,3) circle (\rad) node [anchor=east] {$\cN_1$};
\filldraw[\colo] (1,4) circle (\rad) node [anchor=south] {$\cN_2$};
\filldraw[\colo] (2,3) circle (\rad) node [anchor=west] {$\cN_3$};
\node at (1,1.5) {\Huge $\ast$};

\draw[postaction={decorate}] (1,1.8) --(1,3.7);
\draw[postaction={decorate}] (0.8,3.8)--(0.1,3.1) ;
\draw[postaction={decorate}] (0.2,3) --(1.8,3);
\end{tikzpicture}
\caption{Second way}
\end{subfigure}%
\hspace{7pt}
\begin{subfigure}{0.25\textwidth}
\begin{tikzpicture}[very thick,decoration={
    markings,
    mark=at position 1 with {\arrow{>}}}
    ]
\filldraw[\colo] (0,3) circle (\rad) node [anchor=east] {$\cN_1$};
\filldraw[\colo] (1,4) circle (\rad) node [anchor=south] {$\cN_2$};
\filldraw[\colo] (2,3) circle (\rad) node [anchor=west] {$\cN_3$};
\node at (1,1.5) {\Huge $\ast$};

\draw[postaction={decorate}] (1,1.8) --(1,3.7); 
\draw[postaction={decorate}]  (1.8,3) -- (0.2,3);
\draw[postaction={decorate}] (1.2,3.8)-- (1.8,3.2);
\end{tikzpicture}
\caption{Third way}
\end{subfigure}\\ \ \\ \ \\
\begin{subfigure}{0.25\textwidth}
\begin{tikzpicture}[very thick,decoration={
    markings,
    mark=at position 1 with {\arrow{>}}}
    ]
\filldraw[\colo] (0,3) circle (\rad) node [anchor=east] {$\cN_1$};
\filldraw[\colo] (1,4) circle (\rad) node [anchor=south] {$\cN_2$};
\filldraw[\colo] (2,3) circle (\rad) node [anchor=west] {$\cN_3$};
\node at (1,1.5) {\Huge $\ast$};

\draw[postaction={decorate}] (1.2,1.8) --(1.9,2.85); 
\draw[postaction={decorate}] (0.8,3.8) -- (0.1,3.1); 
\draw[postaction={decorate}]  (1.8,3.2)--(1.2,3.8); 
\end{tikzpicture}
\caption{Fourth way}
\end{subfigure}%
\hspace{7pt}
\begin{subfigure}{0.25\textwidth}
\begin{tikzpicture}[very thick,decoration={
    markings,
    mark=at position 1 with {\arrow{>}}}
    ]
\filldraw[\colo] (0,3) circle (\rad) node [anchor=east] {$\cN_1$};
\filldraw[\colo] (1,4) circle (\rad) node [anchor=south] {$\cN_2$};
\filldraw[\colo] (2,3) circle (\rad) node [anchor=west] {$\cN_3$};
\node at (1,1.5) {\Huge $\ast$};

\draw[postaction={decorate}] (1.2,1.8) --(1.9,2.85);
\draw[postaction={decorate}] (0.1,3.1) --(0.8,3.8);
\draw[postaction={decorate}] (1.8,3)--(0.2,3);
\end{tikzpicture}
\caption{Fifth way}
\end{subfigure}%
\hspace{7pt}
\begin{subfigure}{0.25\textwidth}
\begin{tikzpicture}[very thick,decoration={
    markings,
    mark=at position 1 with {\arrow{>}}}
    ]
\filldraw[\colo] (0,3) circle (\rad) node [anchor=east] {$\cN_1$};
\filldraw[\colo] (1,4) circle (\rad) node [anchor=south] {$\cN_2$};
\filldraw[\colo] (2,3) circle (\rad) node [anchor=west] {$\cN_3$};
\node at (1,1.5) {\Huge $\ast$};

\draw[postaction={decorate}] (0.8,1.8)--(0.1,2.85);
\draw[postaction={decorate}] (0.2,3) --(1.8,3);
\draw[postaction={decorate}] (1.8,3.2)--(1.2,3.8);
\end{tikzpicture}
\caption{Sixth way}
\end{subfigure}
\caption{The six possible ways to visit all the three targets.}\label{figure1}
\end{figure}
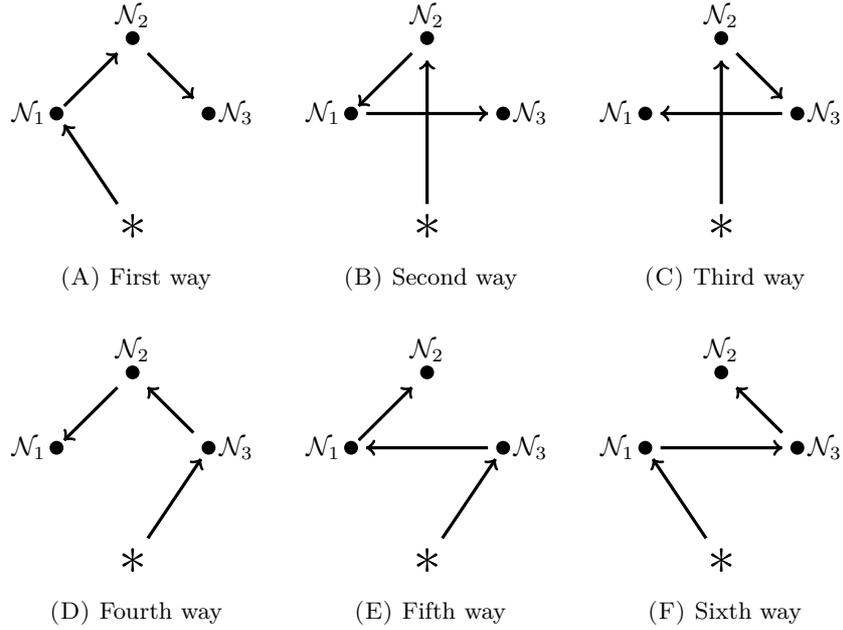

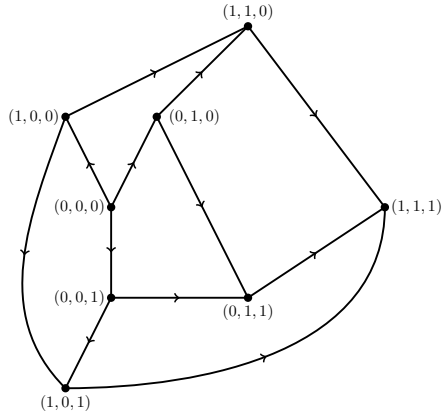
\begin{figure}[h]
\centering
\scalebox{0.6}{
\begin{tikzpicture}[very thick,decoration={
    markings,
    mark=at position 0.5 with {\arrow{>}}}
    ]
\draw[postaction={decorate}] (0,0)--(0,-2);
\draw[postaction={decorate}] (0,0)--(-1,2);
\draw[postaction={decorate}] (0,0)--(1,2);
\draw[postaction={decorate}] (0,-2)--(-1,-4);
\draw[postaction={decorate}] (0,-2)--(3,-2);
\draw[postaction={decorate}] (-1,-4) to[out=0,in=-90] (6,0);
\draw[postaction={decorate}] (-1,2)--(3,4);
\draw[postaction={decorate}] (1,2)--(3,4);
\draw[postaction={decorate}] (1,2)--(3,-2);
\draw[postaction={decorate}] (3,-2)--(6,0);
\draw[postaction={decorate}] (3,4)--(6,0);
\draw[postaction={decorate}] (-1, 2) to[out=-110] (-1, -4);
\filldraw[\colo] (0,0) circle (\rad)  node [anchor=east] {$(0,0,0)$};
\filldraw[\colo] (0,-2) circle (\rad)  node [anchor=east] {$(0,0,1)$};
\filldraw[\colo] (3,-2) circle (\rad)  node [anchor=north] {$(0,1,1)$};
\filldraw[\colo] (-1,-4) circle (\rad)  node [anchor=north] {$(1,0,1)$};
\filldraw[\colo] (-1,2) circle (\rad)  node [anchor=east] {$(1,0,0)$};
\filldraw[\colo] (1,2) circle (\rad)  node [anchor=west] {$\ (0,1,0)$};
\filldraw[\colo] (3,4) circle (\rad)  node [anchor=south] {$(1,1,0)$};
\filldraw[\colo] (6,0) circle (\rad)  node [anchor=west] {$(1,1,1)$};
\end{tikzpicture}
}
\caption{The direct network corresponding to $N=3$ targets.}\label{figure2}
\end{figure}
\end{example}

The possible optimal switching path from $p$ to $\bar p$ must be performed within a fixed final time $T>0$. However here we will assume that an agent at the time $T$ may be still on an intermediate node and then, in that case, it will pay a final cost. Hence, for an agent in the node $p\neq\bar p$ at time $t<T$, the number of the admissible subsequent switches is at most $N-\sum_ip^i\leq N$. The control at disposal of an agent in the node $p$ at time $t$ is then: the number of switches $0\leq r\leq N-\sum_ip^i$; the decision/switching instants $\sigma=(t=t_0< t_1<t_2<\ldots<t_r\leq T)$ and the switching path $\pi$ given by the sequence of the nodes $p=p_0, p_1,\ldots,p_{r}$, satisfying $p_1\in\cI_ {p_0}$, $p_{i+1}\in\cI_{p_i}$, $i=0,1,\ldots,r-1$. We assume that the choice $1\leq r<N-\sum_ip^i$ requires that $t_r=T$ and obviously $p_r\neq\bar p$ (because the number of switches $r$ is not sufficient in order to reach $\bar p$ from $p=p_0$). Moreover, if the choice is $r=0$, then, necessarily, either $p=p_0\neq\bar p$ and $t=t_0=T$ (that is the time is already over) or $p=p_0=\bar p$ and $t=t_0\leq T$ (that is the agent may still have time at disposal but instead no more switches: it is already on $\bar p$). In particular, this implies that an agent can not decide to permanently stand still on a node $p$ along a switching path unless $p=p_r=\bar p$ (or $t_r=T$). To resume, the control at disposal of the agent, which is in $p$ at time $t$, is a triple as
$$
(r, \sigma, \pi)=(r, t_0, t_1,\ldots, t_r, p_0, p_1, \ldots, p_r)
$$
where $t_0=t$ and $p_0=p$. Actually, it is the switching evolution inside the network at disposal of the agent with constraints as specified here above. For example, referring to the network in Figure \ref{figure2}, the following switching evolutions/controls are admissible
\begin{align*}
&\big(2, t_0, t_1, t_2\leq T, (0, 1, 0), (1, 1, 0), (1, 1, 1)\big),\\
&\big(2, t_0, t_1, t_2=T, (0, 0, 0), (1, 0, 0), (1, 0, 1)\big),
\end{align*}
whereas the following ones are not admissible
$$
\big(0, t_0<T, (1, 1, 0)\big),\quad\big(2, t_0, t_1, t_2<T, (0, 0, 0), (1, 0, 0), (1, 0, 1)\big).
$$
In particular, $t_0\ldots, t_{r-1}$ are seen as decision instants and $t_1, \ldots, t_r$ are seen as switching instants. That is the agent at time $t_i\in\{t_0, \ldots t_{r-1}\}$ decides to switch from $p_i$ to $p_{i+1}$ and to perform such a switch at the time $t_{i+1}\in\{t_1, \ldots, t_r\}$. Note that $t_1,\ldots t_{r-1}$ are both decision and switching instants, and this means that the decision about the next switch occurs exactly at the actual switching time.

The cost to be minimized is (note that by the argumentation above if $p\neq\bar p$ and $t<T$, then necessarily $r\geq 1$)
\begin{multline}
\label{costoJnuovo}
J(p, t, (r, \sigma, \pi), \rho)\\
=\begin{cases}\sum_{i=1}^rC(p_{i-1}, p_i, t_{i-1}, t_i, \rho)+\tilde C(p_r, t_r)&\text{if }p\neq\bar p,t<T\\
\tilde C(p, t)&\text{if }p=\bar p\text{ or }(p\neq\bar p, t=T)
\end{cases},
\end{multline}
where:
\begin{itemize}
\item[--] $\rho=\left(\rho_0,\ldots,\rho_{(2^N-1)}\right)\in L^2([0, T], [0, 1])^{2^N}$ is a $(2^N)$-uple of $L^2$ functions $\rho_j:[0, T]\lra[0, 1]$. Here we are using a possible enumeration of the nodes, and every $\rho_j(t)$ represents the mass of the agents at the $j$-node at time $t$. In particular, in the overlying optimal visiting problem, this gives the mass of agents with the same remaining targets to be visited as detected by the positions of the zeros in the string representing the node (we stress here again that $N$ is the number of targets in the overlying optimal visiting problem, and $2^N$ is the number of the $N$-strings of $0$ and $1$ bringing the information about the visited targets and that, in our network problem, represent the nodes of the network itself. See the Introduction and the beginning of Section \ref{optvisprob}.).
\item[--]
\begin{align*}
C:\cD\subset\cI\times\cI\times[0, T]\times]0, T]\times L^2([0, T], [0, 1])^{2^N}&\longrightarrow[0, +\infty[ \\
(p, p', t, \tau, \rho)&\longmapsto C(p, p', t, \tau, \rho)
\end{align*}
(where $\mathcal D$ is such that $(p,p',t,\tau\rho)\in\cal D$ if and only if $p'\in{\cal I}_p$ and $\tau>t$) is the cost function, that is the cost that an agent incurs when, at the (decision) time $t$, being on the node $p$, decides that it will switch to a new node $p'\in\cI_p$ at the (switching) time $\tau>t$. We assume that
\begin{itemize}
\item[$(i)$] For every $(p, p')\in\cI\times\cI_p$ and $\tau\in]0, T]$, the map $(t, \rho)\longmapsto C(p, p', t, \tau, \rho)$ is bounded and Lipschitz continuous in $[0,\tau-h]\times L^2([0,T],[0,1])^{2^N}$, for all sufficiently small $h>0$ and independently of $\tau$, that is, there exists $L>0$, depending only on $h$, such that
$$
|C(p, p', t', \tau, \rho')-C(p, p', t'', \tau, \rho'')|\leq L\left(|t'-t''|+\|\rho'-\rho''\|_{L^2([0, T], [0, 1])}\right);
$$
\item[$(ii)$] For every fixed $\rho$, $(p, p')\in\cI\times\cI_p$ and $t\in[0, T]$, $C$ is decreasing in $\tau\in]t,T]$ and $\lim_{\tau\rightarrow t^+}C(p, p', t, \tau, \rho)=+\infty$, uniformly with respect to $(p,p',t,\rho)$;
\item[$(iii)$] $C(p, p, \cdot, \cdot)=0$ for every $p\in\cI$ and $C(\cdot, \cdot, T, T)=0$. These assumptions correspond to the cases when the agent is on $p=p_r=\bar p$ and $t_r$ is not necessarily $T$ and when $t_r=T$ but the agent is on $p=p_r\neq\bar p$, and moreover give some kind of continuity of \eqref{costoJnuovo}.
\end{itemize}
\item[--] The cost $\tilde C$ is bounded and Lipschitz continuous in time and it represents the final cost that an agent incurs at the end of the switching path $(p_r, t_r)$. For example
\begin{itemize}
\item[--] If $t_r=T$, it depends on the number of the zeros in $p_r$ (that is the number of the remaining targets to be visited);
\item[--] If $p_r=\bar p$, it depends on the remaining time $T-t_r$ (that is the agent is penalized if $\bar p$ is obtained before $T$);
\item[--] If $p_r=\bar p$ and $t_r=T$, then it is null.
\end{itemize}
\end{itemize}
\begin{definition}
\label{optimallygen}
Let $p\in{\cal I}$, $p'\in{\cal I}_p$, and $t<T$ be fixed. We say that the switch from $p$ to $p'$ with decision instant $t$ optimally generates $\tau\in]t,T]$ as switching instant if there exists a control $(\bar r,\bar\sigma,\bar\pi)$, with $\bar r\ge1$, $\bar\sigma=(t_0=t,t_1=\tau,t_2,\dots,t_{\bar r})$ and $\bar\pi=(p_0=p,p_1=p',p_2,\dots,p_{\bar r})$ which minimizes the cost $J$ among all controls $(r,\sigma,\pi)$ such that $r\ge1$, $\sigma=(t_0=t,t_1,\dots,t_r)$, $\pi=(p_0=p,p_1=p',p_2,\dots,p_r)$. In other words: if whenever an agent in $p$ at the time $t$ decides to switch to $p'$ (independently of the optimality of such a choice) then $\tau$ is an optimal choice as switching instant.

We denote by $\varphi_{p,p'}:t\longmapsto\varphi_{p,p'}(t)=\tau$ the function that, for all $p,p'$ fixed, gives, for any $t$, the optimally generated switching $\tau$.
Note that the optimally generated $\tau$ may be not unique and hence the function $\varphi_{p,p'}$ may be multivalued. Also note that $\varphi_{p,p'}$ depends on $\rho$ too. In the sequel, for notational convenience, we will often drop the index ${(p,p')}$ and simply write $\varphi$.
\end{definition}

Moreover, other modeling assumptions are the following:
\begin{itemize}
\item[$(iv)$] If at the decision time $t$, an agent in a node $p$ chooses the switching time $\tau$ in order to switch to $p'$, then, in the time interval $[t, \tau[$, it is assumed that such an agent continues to concur to the total mass present in the node $p$ (coherently with the fact that the switch will occur at time $\tau$ and hence the agent will be on $p$ in the time interval $[t, \tau[$). However, the agent can not change its decision (switching to $p'$ at time $\tau$) or take other decision in the time interval $]t,\tau[$. In other words, in the time interval $]t, \tau[$ it must stay on $p$;
\item[$(v)$] For the switching from $p$ to $p'$, if we have two different decision times $t_1$, $t_2$ with $t_1<t_2$, which optimally generate the switching times $\tau_1, \tau_2<T$ respectively (see Definition \ref{optimallygen}), then $\tau_1<\tau_2$.
\end{itemize}
Assumption $(iv)$ suggests the following useful definition.
\begin{definition}
An agent which is in $p$ at time $t$ and uses the control
$$
(r, t_0=t, t_1,\ldots, t_r, p_0=p, p_1, \ldots, p_r)
$$
is called a decision-making agent at the decision instants $t_0,\ldots,t_{r-1}$. Actually, since there is no incoming flow in our network (all the agents are already present at $t=0$), all the agents are decision-making at $t=0$. In particular, any single agent will take a new decision, mandatory, at time $\tau$ when it will switch to the new node; in other words: all agents are decision-making at $t = 0$ and
they will return to be decision-making again exactly when, and only when, they switch to a new node.
\end{definition}
\begin{remark}
\label{unicoarrivo}
Assumption $(ii)$ means that, if the switching time is too much close to the corresponding decision time, then the agent pays an high cost.

The second part of assumption $(iv)$ (the agent can not change the decision in $[t, \tau[$) is certainly due to the discrete feature of the time-dependent component $\sigma$ of the global control $(r,\sigma,\pi)$, but it may also be justified by the overlying optimal visiting problem where, when an agent is moving from one target to another then, under some assumptions, it is not optimal to change destination or to come back to the previous node (see \cite{BFMP}). Also, the interpretation as job scheduling may justify such an assumption.

From assumption $(v)$, it follows $(v')$: any optimal switching time less than $T$ originates from a unique decision time. This can be also directly proved by assuming further hypotheses (see Remark \ref{esempiotempi}). Moreover, suppose that the decision time $t$ optimally generates the switching times $\tau_1$, $\tau_2$ with $\tau_1<\tau_2$ for the switching from $p$ to $p'$. Then, in view of assumption $(v)$, in the time interval $[\tau_1, \tau_2[$ only the agents with decision time $t$ can switch from $p$ to $p'$. More generally, if we define $\tau^-:=\inf_{\tau}\{\text{$\tau$ is optimal for $t$}\}$ and $\tau^+:=\sup_{\tau}\{\text{$\tau$ is optimal for t}\}$, in the time interval $[\tau^-, \tau^+[$, only the agents with decision time $t$ can switch from $p$ to $p'$. Hence, we can consider the function $\phi:t\longmapsto\tau$, giving the optimal switching instant $\tau$ for the decisional instant $t$, as a maximal monotone graph filling the jumps by vertical segments, and so, in this case, $\phi$ is a multivalued function.

All the previous assumptions and arguments can be justified by a possible overlying optimal visiting problem with suitable energy and congestion costs (see \cite{BFMP}). See also Remark \ref{esempiotempi}.
\end{remark}
The value function of the problem is
\begin{equation}
\label{valuef}
V(p, t, \rho)=\inf_{(r, \sigma, \pi)}J(p, t, (r, \sigma, \pi), \rho)
\end{equation}
and a control $(r, \sigma, \pi)$ is said to be optimal for $(p, t)$ if $V(p, t, \rho)=J(p, t, (r, \sigma, \pi), \rho)$.
\begin{definition}
\label{optimalgen2}
Let $p\neq\bar p$, $t\in[0,T[$ and $\tau\in]t,T]$ be fixed. We say that $\tau$ is optimal for $V(p,t, \rho)$ if there exists a control $(\bar r,\bar\sigma,\bar\pi)$ with $\bar r\ge1$, $\bar\sigma=(t_0=t,t_1=\tau,t_2,\dots,t_{\bar r})$ and $\bar\pi=(p_0=p,p_1,p_2,\dots,p_{\bar r})$ which is optimal, that is minimizes the cost $J$ among all controls. In other words, there exists an optimal control whose first switching instant is $\tau$.
\end{definition}
Given next Proposition \ref{dynprogsw} (and in particular looking at its proof), the previous definition is equivalent to require that there exists $p'\in{\cal I}_p$ such that the pair $(p',\tau)$ realizes the minimum in
$$
V(p,t, \rho)=\inf_{\substack{p'\in\cI_{p}\\ \tau\in]t, T]}}\left\{V(p',\tau)+C(p,p',t,\tau,\rho)\right\}.
$$
\section{The optimal switching problem with fixed mass $\rho$}
\label{fixedmass}
In this section, we mostly assume that the mass $\rho\in L^2([0, T], [0, 1])^{2^N}$ is a priori fixed and then, when not needed, we do not display it as entry of the cost $J$ and of the value function $V$.
\begin{proposition}
\label{lipcontV}
For all $k>0$, the value function $V$ in \eqref{valuef} is bounded and Lipschitz continuous in $[0,T-k]$, uniformly in $\rho$. Moreover, if $\rho^n$ converges to $\rho$ in $L^2([0, T], [0, 1])$, then $V(p,\cdot,\rho^n)$ uniformly converges to $V(p,\cdot,\rho)$ on $[0,T-k]$, for all $p$. Also, if $t'^n$ is optimal for $V(p,t^n,\rho^n)$ and $t'^n$, $t^n$ converge to $t'$, $t<T$ respectively, then $t'$ is optimal for $V(p,t,\rho)$ (see Definition \ref{optimalgen2} for $t$ optimal).
\end{proposition}
\begin{proof}
Using $(i)$ and $(ii)$ in Section \ref{optvisprob}, for all $t\in[0,T-k]$, it is, for all $p\neq\bar p$ and $p'\in{\cal I}_p$, $V(p,t,\rho)\le C(p,p',t,T,\rho)+\tilde C(p',T)\le C(p,p',t,t+k,\rho)+\tilde C(p',T)$, and this gives the boundedness. We take $0<h<k$ such that, for all $t\in[0,t_k]$, $p,p',\rho$ it is, for $t<\tau\le t+h$,
$$
C(p,p',t,\tau,\rho)>\left(\max_{p,p',t\in[0,T-k],\rho}C(p,p',t,t+k,\rho)+\tilde C(p',T)\right)
$$
and note that such $h$ exists for hypothesis $(ii)$ in Section \ref{optvisprob}. We then get that if the decisional instant $t\in[0,T-k]$ optimally generates the switching instant $\tau$ then $\tau\ge t+h$. Take $p$ and  $t',t''\in[0,T-k]$ such that $|t'-t''|<h$. For $\varepsilon>0$ let $(r,\sigma,\pi)$ be such that $V(p,t'')\ge J(p,t'',(r,\sigma,\pi))-\varepsilon$. Hence the control triple $(r,\sigma,\pi)$ is also admissible for $t'$ (all the instants in $\sigma$ are larger than $t'$, because their distance from $t'$ is at least $h$). We have
\begin{align*}
V(p,t')-V(p,t'')&=V(p, t')-V(p, t'')\leq J(p, t', (r, \sigma, \pi))-J(p, t'', (r, \sigma, \pi))+\ep\\
&=C(p, p_1, t', t_1)+\sum_{i=2}^rC(p_{i-1}, p_i, t_{i-1}, t_i)+\tilde C(p_r, t_r)\\
&\quad -C(p, p_1, t'', t_1)-\sum_{i=2}^rC(p_{i-1}, p_i, t_{i-1}, t_i)-\tilde C(p_r, t_r)+\ep\\
&=C(p, p_1, t', t_1)-C(p, p_1, t'', t_1)+\ep\leq L|t'-t''|+\ep,
\end{align*}
where $L$ is the Lipschitz constant of the cost $C$ (see assumption $(i)$), which is independent of $\rho$. By the arbitrariness of $\ep$, the compactness of $[0,T-k]$, and changing the role of $t'$ and $t''$, we get the Lipschitz continuity of $V$ in $[0,T-k]$, uniformly in $\rho$.

For the convergence of $V(p, \cdot, \rho^n)$, note that, by the previous points and by Ascoli-Arzel\`a Theorem, at least for a subsequence, we have the uniform convergence on $[0,T-k]$ to a limit function $\tilde V$. Taking $h>0$ as above (and hence, for all $t\in[0,T-k]$, the optimal $t'$ belongs to $[t+h,T]$), by the Lipschitz continuity hypotheses on $C$ and $\tilde C$ (in particular the continuity of $C$ with respect to $\rho\in L^2$), we get the point-wise convergence to $V(p,\cdot,\rho)$ in $[0,T-k]$, which then turns out to be the uniform limit $\tilde V$, independently of the subsequence. The final point on $t'^n, t^n$ and $t', t$ also comes because, being $t<T$, there exists $k>0$ such that, for large $n$, both $t,t^n$ belong to $[0,T-k]$ and, for example, we can use the characterization of $V$ by Proposition \ref{dynprogsw} which is independent of Proposition \ref{lipcontV}.
\end{proof}
\begin{proposition}
\label{dynprogsw}
The value function $V$ is the unique solution of the following
\begin{equation}
\label{hjsystem}
\begin{cases}
V(p, t)=\inf_{\substack{p'\in\cI_{p}\\t'\in]t, T]}}\{V(p', t')+C(p, p', t, t')\},&(p, t)\in(\cI\setminus\{\bar p\})\times[0, T[\\
V(\bar p, t)=\tilde C(\bar p, t),&t\in[0, T]\\
V(p, T)=\tilde C(p, T),&p\in\cI
\end{cases}.
\end{equation}
\end{proposition}
\begin{proof}
First of all, let us note that the second and third equalities come from the definition of $J$ \eqref{costoJnuovo}. We have to prove the first equality. Suppose that $p_r=\bar p$, that is $p_r=(1, 1, \ldots, 1)$.

\smallskip
\noindent \textbf{Case $1.$} Let $p\in\cI$ be such that $\sum_ip^i=N-1$, for instance $p=(1, 1, \ldots, 1, 0)$, so $r=1$, $\pi=(p, \bar p)$ and $\sigma=(t, t')$ for some arbitrary $t'\in]t, T]$. Thus we have to prove that
\begin{equation}
\label{foruniqueness}
V(p, t)=\inf_{t'\in]t, T]}\{V(\bar p, t')+C(p, \bar p, t, t')\}=\inf_{t'\in]t, T]}[C(p, \bar p, t, t')+\tilde C(\bar p, t')]
\end{equation}
since $V(\bar p, \cdot)=\tilde C(\bar p, \cdot)$. The last term in the above equality is
$$
\inf_{(r, \sigma, \pi)}J(p, t, (r, \sigma, \pi))=V(p, t)
$$
being the controls $(1, (t, t'), (p, \bar p))$ the only admissible ones for $(p, t)$.
\smallskip
\\
\noindent \textbf{Case $2.$} Let $p\in\cI$ be such that $\sum_ip^i=N-2$, that is, for instance, $p=(0, 0, 1,\ldots, 1)$. In this case, the admissible controls must have either $r=2$ or $r=1$, and so $V$ is the minimum of the infimum of the cost over the controls with $r=2$ and the infimum of the cost over the controls with $r=1$. In the first case, setting $t_{r-2}=t$ and $p_{r-2}=p$, we have
\begin{equation*}
\resizebox{1\hsize}{!}{$%
\begin{split}
V(p_{r-2}, t_{r-2})&=\inf_{\substack{(r, \sigma, \pi)\text{ s.t.}\\r=2\\ \sigma=(t_{r-2}, t_{r-1}, t_r)\\ \pi=(p_{r-2},p_{r-1}, \bar p)}}J(p, t, (r, \sigma, \pi))\\
&=\inf_{\substack{t_{r-1}\in]t_{r-2}, T[\\t_r\in]t_{r-1}, T]\\p_{r-1}\in\cI_{p_{r-2}}}}\left[C(p_{r-2}, p_{r-1}, t_{r-2}, t_{r-1})+C(p_{r-1}, \bar p, t_{r-1}, t_r)+\tilde C(\bar p, t_r)\right]\\
&=\inf_{\substack{t_{r-1}\in]t_{r-2}, T[\\p_{r-1}\in\cI_{p_{r-2}}}}\Bigg[C(p_{r-2}, p_{r-1}, t_{r-2}, t_{r-1})\\
&\quad +\inf_{t_r\in]t_{r-1}, T]}\left[C(p_{r-1}, \bar p, t_{r-1}, t_r)+\tilde C(\bar p, t_r)\right]\Bigg]\\
&=\inf_{\substack{t_{r-1}\in]t_{r-2}, T[\\p_{r-1}\in\cI_{p_{r-2}}}}\left[V(p_{r-1}, t_{r-1})+C(p_{r-2}, p_{r-1}, t_{r-2}, t_{r-1})\right],
\end{split}$%
}
\end{equation*}
where the last equality comes from Case $1)$. The desired result follows.

In the second case, $r=1$, we must necessarily have $p_r\neq\bar p$ and $t_r=T$. Thus we have only to prove that
$$
V(p, t)=\inf_{p_r\in\cI_p}\{V(p_r, T)+C(p, p_r, t, T)\}=\inf_{p_r\in\cI_p}\{C(p, p_r, t, T)+\tilde C(p_r, T)\}
$$
since $V(\cdot, T)=\tilde C(\cdot, T)$. The last term in the above equality is
$$
\inf_{(r, \sigma, \pi)}J(p, t, (r, \sigma, \pi))=V(p, t)
$$
being, in this case, the controls $(1, (t, T), (p, p_r))$ the only ones we are taking account of.

Up to now, we proved the equality for every $(p, t)$ such that $\sum_ip^i=N-1$ and $\sum_ip^i=N-2$. Proceeding backwardly in this way we then can prove all the other cases with $\sum_{i}p^i=N-s$ for $s=3,\ldots,N$.

Still arguing backwardly, the uniqueness comes from the fact that any other function satisfying \eqref{hjsystem}, by \eqref{foruniqueness} must coincide with $V$ on the nodes $(p, t)$ with $t\in[0, T[$ and $p$ such that $\sum_ip^i=N-1$.
\end{proof}
\begin{remark}
\label{esistenzah}
The infimum in the first line of \eqref{hjsystem} is indeed a minimum realized for some $t'$ belonging to $[t+h, T]$, where $h$ is as in the proof of Proposition \ref{lipcontV} for some $k$ such that $t\in[0,T-k]$. Indeed, the quantity inside the minimization is continuous in $[t+h,T[$ and tends to $+\infty$ for $t'\to T^-$.

We also note that, considering $t=0$, any optimal control $(r,\sigma,\pi)$ for $V(p,0)$ with $p\neq\bar p$, is such that $t_{i+1}-t_i\ge\bar h$ for a suitable $\bar h>0$ independent on $p$ and on $\rho$. This can be seen as in the proof of Proposition \ref{lipcontV}, observing that $V(p,0)\le C(p,p',0,T,\rho)+\tilde C(p',T)$. The presence of this sort of minimal waiting time $\bar h$ between two consecutive switches, when starting at $t=0$, will lead to a piecewise continuous/constant feature of the evolution of the masses $\rho$ with a uniform bounded number of pieces in $[0, T]$. Also note that, the existence of such $\bar h$ gives the conclusion that for all possible optimal control $(r,\sigma,\pi)$ for an agent at $t=0$, all the decisional instants will belong in $[0,T-\bar h]$. Hence, in the sequel, we will actually work (see for example (9)) with pairs $(p,t)$ such that, whenever $p\neq\bar p$ and $t<T$ (otherwise the agent would have finished its evolution), it is necessarily $t\in[0,T-\bar h]$, where the value function is bounded and Lipschitz continuous, by Proposition \ref{lipcontV}.
\end{remark}
Also justified by Remark \ref{esistenzah}, we define, for $p\neq\bar p$ and $t<T$,
\begin{equation}
\label{funzioniP}
P(p, t)=\argmin_{\substack{p'\in\cI_p\\t'\in]t, T]}}\{V(p', t')+C(p, p', t, t')\}.
\end{equation}
In other words, $P(p, t)$ is the couple $(p', t')$ of the node $p'$ where it is optimal to switch at the switching instant $t'>t$. As above, we do not display the dependence on $\rho$.
\begin{remark}[still on assumption $(v)$ in \S2]
\label{esempiotempi}
Assumption $(v)$ may hold for example in the case where the cost $C$, besides $(ii)$, is derivable w.r.t. the switching time-variable $\tau$ with derivative $C_\tau$ strictly increasing w.r.t. the quantity $\tau-t$. A possible cost satisfying the previous hypotheses may be for example of the form
\begin{equation}
\label{eq:esempio_C}
C(p, p', t, \tau, \rho)=\frac{\bar C(p, p', \rho)}{\tau-t}.
\end{equation}
Moreover we assume that $V$ is convex in $[0,T[$ (in Appendix \ref{appendixB} we give an explicit example where $V$ is convex). It follows that it is two times derivable in time almost everywhere (see for example \cite{EG}, Theorem 1, p. 242). For the following counterexample, we are going to assume that the first derivative exists everywhere.
By contradiction, let us suppose that if, for the switching from $p$ to $p'$, the decision times $t_1$, $t_2$ with $t_1<t_2$ optimally generate the switching times $\tau_1, \tau_2<T$ respectively, then $\tau_2<\tau_1$. Hence it follows that $\tau_1>\tau_2\geq t_2>t_1$. This means that
$$
\inf_{\tau\geq t_1}\left\{V(p', \tau, \rho)+C(p,p',t_1,\tau,\rho)\right\}=V(p', \tau_1, \rho)+C(p,p',t_1,\tau_1,\rho),
$$
$$
\inf_{\tau\geq t_2}\left\{V(p', \tau, \rho)+C(p,p',t_2,\tau, \rho)\right\}=V(p', \tau_2, \rho)+C(p,p',t_2,\tau_2, \rho).
$$
First order conditions give
$$
V'(p', \tau_1, \rho)+C_\tau(p,p',t_1,\tau_1, \rho)=0,
$$
$$
V'(p', \tau_2, \rho)+C_\tau(p,p',t_2,\tau_2, \rho)=0.
$$
Therefore
$$
V'(p', \tau_1, \rho)=-C_\tau(p,p',t_1,\tau_1, \rho)<-C_\tau(p,p',t_2,\tau_2, \rho)=V'(p', \tau_2, \rho),
$$
which contradicts the convexity of $V$ in time.

Recall that assumption $(v)$ implies $(v')$: any optimal switching time less than $T$ originates from a unique decision time. With the same hypotheses on $C$ as above, $(v')$ can be also inferred, without assuming $(v)$, just assuming that $V$ is derivable w.r.t. the time-variable without any convexity property. Indeed, suppose that at the decision times $t_1, t_2$, $t_1<t_2$, the agents are optimally switching from $p$ to $p'$ with the same switching time $\tau<T$. Arguing as above, with $\tau_1=\tau_2=\tau$, we obtain
$$
V'(p', \tau, \rho)=-C_\tau(p,p',t_1,\tau, \rho)=-C_\tau(p,p',t_2,\tau, \rho),
$$
contradicting $t_1\neq t_2$.

Without the convexity assumption on $V$, we can infer property $(v)$ by $(v')$ if, besides the derivability of $V$, we assume that the map $\varphi:t\longmapsto\tau$ is continuous. Note that, by definition of the optimal switching instant, $\varphi(t)\to T$ as $t\to T$. By contradiction, suppose that if, for the switching from $p$ to $p'$, the decision times $t_1, t_2$ with $t_1<t_2$ optimally generate the switching times $\tau_1, \tau_2<T$ respectively, then $\varphi(t_2)=\tau_2<\tau_1=\varphi(t_1)<T$. Hence, the function $\varphi$ is somehow decreasing in $[t_1,t_2]$ but, by continuity and the limit property above, we must have the existence of $t'\neq t''$ such that $\tau=\varphi(t')=\varphi(t'')$, contradicting $(v')$.

Finally, for what concerns the convexity of $V$, note that if $C$ is strictly convex in $t$ and $\tilde C$ is decreasing in time, due to the decreasingness of $C$ with respect to $\tau$ ($(ii)$) (and the example in \eqref{eq:esempio_C} satisfies both hypotheses), then for all $p$ with one $0$ only (i.e., directly linked to the destination $\bar p$), $V(p,\cdot, \rho)$ is strictly convex, the functions $\phi$ are constantly equal to $T$. Proceeding backwardly, we can then prove that for all the other nodes the value functions are all strictly convex and, in particular, the function $\phi$ is single-valued and increasing (see the example in Appendix \ref{appendixB}).
\end{remark}
\begin{remark}
Let us note that equation \eqref{hjsystem} is in some sense the Dynamic Programming Principle for the value function $V$. However, we can not differentiate it in the time-variable $t$ and obtain an Hamilton-Jacobi-Bellman equation because our model does not take account of a continuous dynamic evolution of the agents.
\end{remark}

\section{On the continuity equations for the flow}
\label{flowsofag}
For what concerns the $\rho$ functions for the masses, using the same possible enumeration of nodes as in \S2, for every $j=0,\ldots,2^N-1$ we will have, at least formally, a system of $2^N$ continuity equations in the variables $\rho_j^{\text{dm}}$, the mass of decision-making agents (see assumption $(iv)$ in \S2), to be interpreted in a suitable formulation that we will see later:
\begin{equation}
\label{sistema2}
\begin{cases}
(\rho_j^{\text{dm}})'(t)=\sum_{p_k|p_j\in\cI_{p_k}}\lambda_{k, j}(s(t),t)\rho_k^{\text{dm}}(s(t))\d_t\\ \hspace{5.3cm}-\sum_{p_h|p_h\in\cI_{p_j}}\l_{j,h}(t, \varphi(t))\rho_j^{\text{dm}}(t)\d_t,&t\in]0, T]\\
\rho_j^{\text{dm}}(0)=\rho_j^0
\end{cases},
\end{equation}
where $\rho_j^0$ is fixed for every $j$, $\phi$ is the (possibly multivalued) function introduced in Remark \ref{unicoarrivo} and $t\longmapsto s(t)\in[0, T]$ takes into account the decision instant $s$ at which an agent switches from $p_i$ to $p_j$ at the switching time $t$. By assumption $(v)$, $s(t)$ is continuous and non-decreasing (being the inverse of the function $\phi$ in Remark \ref{unicoarrivo}) and satisfies $s(t)\leq t$ for every $t$ and $s(0)=0$. Formally such a function $s$ (as well as $\varphi$) should be indexed by $i, j$ but for simplicity we omit that. The first term in the right-hand side of \eqref{sistema2} represents the mass of decision-making agents arriving to $p_j$ at the switching instant $t$ and the second one, the mass of decision-making agents leaving $p_j$ at the decisional instant $t$. The unknowns are the $2^N$ functions $\rho_j^{\text{dm}}$ and the functions $\lambda_{k, j}:{\cal A}\subset[0, T]\times]0, T]\longrightarrow[0, 1]$, $(s,t)\longmapsto\l_{k, j}(s,t)$ (${\cal A}=\left\{(s,t):s<t\right\}$), which indicate how many decision-making agents, in $p_k$ at time $s$, have chosen $P(p_k, s)=(p_j, t)$, \eqref{funzioniP}, that is the percentage of mass of decision-making agents which is in $p_k$ and at time $s$ optimally decides to switch to $p_j$ at $t>s$. Of course, if $\l_{k, j}(s,t)>0$, then, at time $s$, deciding to switch from $p_k$ to $p_j$ at time $t$ is optimal, and we also have $\sum_{p_j|p_j\in\cI_{p_k}}\l_{k, j}(s,\xi)=1$, where $\xi\in\varphi(s)$ is any possible selection for the switch from $p_k$ to $p_j$. Similarly for $\l_{j, h}$.

Note that the previous sum equal to $1$ means that every instant $s$ is a decisional instant for all the decision-making agents present on the node. The fact that those $\lambda$ activate a real switch obviously depends on the real presence of decision-making agents on the node at the time $s$. Indeed, roughly speaking, the interpretation of \eqref{sistema2} is the following one. The functions $\l_{i, j}$, for every $i, j$, give the right way to interpret it. Such functions are basically values between $0$ and $1$ along the curve $t\longmapsto (s(t), t)$, that is $\l_ {i ,j}$ is concentrated on the curve and it is elsewhere null. From a distributional point-of-view, $\l_{i, j}$ is a concentration of Dirac deltas on that curve. In other words, if at the switching instant $t$ the switches from $p_k$ to $p_j$ and from $p_j$ to $p_h$ are both optimal, then $\l_{k, j}$ and $\l_{j, h}$ are possibly nonzero at $(s(t), t)$ and consequently activate the Dirac deltas, which give the corresponding accumulation of mass (of decision-making agents only) on the arrival node at time $t$. In particular, we stress further that the functions $\lambda_{.,.}$ are directly linked to the optimization problem, and hence to the optimal switching function $\varphi$: if to decide at the time $s$ to switch from $p_i$ to $p_j$ at the time $t>s$ is not an optimal choice, then $\lambda_{i,j}(s,t)=0$. See also the fourth line of the system (\ref{mfgsys}) and point $(v)$ in the next subsection. We also point out that such functions $\lambda_{.,.}$, as they appear in (\ref{sistema2}),  look like as unknown, but, as  in the standard fixed-point procedure in mean field games, they will be assumed as known, when one addresses the continuity equation. Indeed, such a standard fixed point procedure, in our case will be (see the next section): take $\rho$, derived the corresponding $\lambda_{.,.}$ by the optimization problem, put such $\lambda_{.,.}$ in (\ref{sistema2}) and calculate the corresponding solution $\tilde\rho$. Finally, ask $\tilde\rho=\rho$ and that will be an equilibrium. The role of the function $\lambda_{.,.}$ in (\ref{sistema2}) is the same role of the optimal field $-\nabla u$ in the Fokker-Planck equation in the standard mean field game, where $u$ is the value function. Of course in our problem, due to the possible lacking of the uniqueness of the optimal control (in the present notation, optimal $p_j$ and optimal $t$), we may have different choices for the functions $\lambda_{.,.}$, but the existence of a suitable choice for having an equilibrium will be guaranteed by a multi-valued fixed point procedure after a suitable convexification.

In the case when the function $t\longmapsto\tau=\varphi(t)$ (Remark \ref{unicoarrivo}) is always a singleton, i.e. not multivalued, then system \eqref{sistema2} may be also interpreted as system of impulsive delayed equations (see for instance \cite{impulsive}). The solutions $\rho_j^{\text{dm}}$ are somehow collections of possibly nonzero values on switching (incoming as well as outgoing) instants, and equal to zero elsewhere. The real mass evolution $\rho_j$, taking into account both decision-making and non-decision making agents, is just the right-continuous constant interpolation of those values. In other words, the $2^N$ solutions $\rho_j$ are constructed node-by-node for every switching time according to the $\l$ functions, and this process gives piecewise constant functions on $[0, T]$ (see also Remark \ref{esistenzah}).

In the next section we are going to make a suitable approximation of the problem, in order to be able to work with piecewise constant functions. Moreover, in that case, we will see a possible direct construction of such functions $\l$ also explaining their presence and roles in \eqref{sistema2}, and then the construction of the functions $\rho$. Actually, we will not use the formal equations \eqref{sistema2} but directly construct step-by-step (switch-by-switch) the solutions. In Figure \ref{figure3}, \S\ref{esistenzamfgapp}, we graphically represent the construction of a possible $\rho^{\text{dm}}$ and its constant interpolation $\rho$. The fact that in the following analysis we will not use directly the equation (\ref{sistema2}) is coherent with the theory of the standard first order mean-field game, where, for the continuity equation one has a natural candidate for the solution: the push-forward of the initial measure via the optimal field $-\nabla u$. Indeed, we are going to directly construct a sort of optimal push-forward of the initial distribution of the agents: an optimal flow inside the network.

\section{The approximated mean-field problem}
\label{apprmfg}
As argued at the end of the previous section, we are going to make a suitable approximation in order to allow us to look for solutions $\rho$ of \eqref{sistema2} in $\cP\cC([0, T], [0, 1])^{2^N}$, where $\cP\cC([0, T], [0, 1])$ is the set of piecewise constant functions from $[0, T]$ to $[0, 1]$. In order to possibly simplify the notation, using the same enumeration of the nodes in \S\ref{optvisprob}, we consider all the functions $\rho_j$ as forming a unique function in a juxtaposed sequence of $2^N$ intervals of length $T$. We then define $\cB:=\cP\cC([0, 2^NT], [0, 1])$ whose elements $\rho$ are still thought as $(\rho_0,\ldots, \rho_{2^{N}-1})$. The mean-field game system we are going to study is formally described by
\begin{equation}
\label{mfgsys}
\begin{cases}
V(p, t, \rho)=\inf_{\substack{p'\in\cI_{p}\\t'\in]t, T]}}\{V(p', t', \rho)+C(p, p', t, t', \rho)\},\hspace{0.58cm}(p, t, \rho)\in\cI\times[0, T[\times\cB\\
V(\bar p, t, \rho)=\tilde C(\bar p, t),\hspace{6.28cm} (t, \rho)\in[0, T]\times\cB\\
V(p, T, \rho)=\tilde C(p, T),\hspace{6.55cm} (p, \rho)\in\cI\times\cB\\
\l_{i,j}(s,t)=0\ \mbox{if }(p_j,t)\not\in P(p_i,s),\\
(\rho_j^{\text{dm}})'(t)=\sum_{p_k|p_j\in\cI_{p_k}}\lambda_{k, j}(s(t), t)\rho_k^{\text{dm}}(s(t))\d_t\\
\hspace{4.98cm}-\sum_{p_h|p_h\in\cI_{p_j}}\l_{j,h}(t, \varphi(t))\rho_j^{\text{dm}}(t)\d_t,\hspace{0.3cm}t\in[0, T]\\
\rho_j^{\text{dm}}(0)=\rho_j^0,\\
\rho_j\ \text{constant interpolation of } \rho_j^{\text{dm}}.
\end{cases}
\end{equation}
Note that the fourth line of \eqref{mfgsys} stands for the fact that if a switch is not optimal, then the corresponding fraction $\lambda$ is zero: no one is following that switch. In particular, we point out that the four lines of (\ref{mfgsys}) take account of the optimization process for the single agent, whereas the last three are the continuity (flow) equation for the mass of agents through the network. We also recall here what was said at the end of the last section for equation (\ref{sistema2}) that also holds for the last three lines of (\ref{mfgsys}): actually, we are not directly use the continuity equation, but directly construct step-by-step the flow, representing the solution.

Next section is devoted to prove the existence of a solution $(\rho_j,\lambda_{j,k})$ of an approximated version of \eqref{mfgsys} and hence of an $\varepsilon$-approximated equilibrium of the mean-field game. Such an approximation is mainly consistent in a suitable approximation of the function $P$ in (\ref{funzioniP}).

\subsection{Existence of an $\varepsilon$-approximated mean-field equilibrium}
\label{esistenzamfgapp}
As usual, we are going to identify the solution $\rho$ of \eqref{mfgsys} as a fixed point of a suitable function. At first sight, given also Remark \ref{esistenzah}, the space where to search for a fixed point would seem to be the following one:
$$
X=\{\rho\in\cB:\text{$\rho$ has at most $M$ pieces of constancy}\},
$$
where $M$ is a priori fixed, for example $M=\left(\frac{2^NT}{h}\right)^{2^N}$. Note that such a space can be made compact with respect to a suitable convergence but it is certainly not convex (every $\rho$ has different pieces from the others) and, to perform a fixed-point procedure, we need that $X$ satisfies a convexity property. Therefore, to overcome this difficulty, we fix $\varepsilon>0$ and we consider the partition $\cP_{\varepsilon}$ of $[0, 2^NT]$, given by the nodes $0<\varepsilon<2\varepsilon<\ldots\leq2^NT$ with $\ep=\frac{T}{m}$ for some $m\in\N$. We then consider the space
\begin{multline*}
C_{\varepsilon}=\left\{\rho\in L^2([0, 2^NT], [0, 1]):\right.\\
\left.\rho\text{ is piecewise constant on the open intervals of $\cP_{\ep}$ and $\|\rho\|_{\infty}\leq\|\rho_0\|_{\infty}$}\right\}.
\end{multline*}
Now, $C_{\varepsilon}$ is convex and compact with respect to the $L^2$ topology. Indeed, since the partition $\cP_{\ep}$ is fixed and all the functions $\rho$ are constant on it, from every interval of $\cP_{\ep}$ we can extract a convergent constant subsequence whose limit belongs to $L^2$.

We then look for a fixed point of a suitable multi-function $\psi_{\varepsilon}:C_{\varepsilon}\lra\cP(C_{\varepsilon})$, $\rho\longmapsto\psi_{\varepsilon}(\rho)$, that is we look for $\rho_{\ep}\in C_{\ep}$ such that $\rho_{\ep}\in\psi_{\ep}(\rho_{\ep})$. Roughly speaking, the idea is to construct $\psi_{\ep}$ as follows:
\begin{itemize}
\item[$(i)$] $\rho$ is put into \eqref{hjsystem} and the value function $V$ is derived;
\item[$(ii)$] $V$ is inserted in \eqref{funzioniP} and the variable $P$, which is not necessarily unique (that is, a priori, there may exist more than one optimal switching instant and more than one admissible subsequent node where it is optimal to switch), is derived;
\item[$(iii)$] We suitably approximate the optimal switching instants given by $P$ at point $(ii)$ with the nodes of the partition $\cP_{\ep}$;
\item[$(iv)$] With such approximated $\varepsilon$-optimal variables $P_{\ep}$ as in $(iii)$, we construct all the possible optimal switching paths with their decision and switching times;
\item[$(v)$] For each optimal switching path $\pi$ of point $(iv)$, we construct the corresponding functions $\l$ in \eqref{sistema2}, as all the agents were following $\pi$, that is
$$
\l^{\pi, \ep}_{i, j}(s,t)=\begin{cases}1,&(p_j,t)\in P_{\ep}(p_i, s)\cap\pi\\
0,&\text{otherwise}
\end{cases};
$$
where the notation $(p_j,t)\in P_{\ep}(p_i, s)\cap\pi$ means that, being at $p_i$ at time $s$, the choice of switching to $p_j$ at the time $t>s$ is $\varepsilon$-optimal, in the sense as explained in the previous point $(iv)$.
\item[$(vi)$] For any $\pi$, we insert the functions $\l^{\pi}$ into \eqref{sistema2}, obtaining the evolution of the mass $\rho^{\pi}\in C_\varepsilon$;
\item[$(vii)$] By a suitable convexification (interval by interval of the partition $\cP_{\ep}$) of the functions $\rho^{\pi}$ of $(vi)$, we construct a set of functions $\psi_{\ep}(\rho)$, which is contained in $\cP(C_{\varepsilon})$;
\item[$(viii)$] By proving that $\psi_{\ep}(\rho)$ is a non-empty and convex subset of $C_{\ep}$ and that the map $\rho\longmapsto\psi_{\ep}(\rho)$ has closed graph, we can apply the fixed-point Kakutani-Ky Fan Theorem (see for example \cite{KAK}) to find a desired $\rho_{\ep}$.
\end{itemize}
Note that, by construction, $\rho_{\ep}$, together with the coefficients $\l$ of the convex combinations of the extremal $\rho^{\pi}$ as in point $(vii)$, gives what can be considered as an approximated solution of \eqref{mfgsys} and hence an $\ep$-mean-field equilibrium.
\begin{definition}
\label{def:epsilon_equilibrium}
An $\varepsilon$-mean-field equilibrium of problem (\ref{mfgsys}) is a fixed point $\rho_\varepsilon$ of the multi-valued map $\psi_\varepsilon$: $\rho_\varepsilon\in\psi_\varepsilon(\rho_\varepsilon)$.
\end{definition}
We divide the construction of $\psi_{\ep}(\rho)$ into some steps. A general definition for the multi-function $\psi_\varepsilon$, covering all possible cases and, in particular, all possible networks, starting from the optimal visiting problem with $N$ targets, is beyond the purpose of this paper and, probably, it is not helpful and even meaningless. In fact, the complexity of the model increases drastically, both from a notational point-of-view and for the number of cases and subcases to be considered. Hence we show the construction of it for some examples. This, however, does not weaken the proofs of the results. In the following, for simplicity, we suppose $N=3$ (compare with Figure \ref{figure2}) and consider only paths starting from $p_0=(0,0,0)$ and that, at the initial time $t=0$, all the agents are in $p_0$. Moreover note that all the paths start at time $t=0$, but this is implicit, in our model. In Remark \ref{generalsituation} below a more general situation is considered. Moreover, the example in Appendix \ref{appendix} may be also somehow enlightening. Anyway, at the end of Remark \ref{generalsituation} we briefly give a possible (certainly non-operative) definition of $\psi_\varepsilon$.

In the sequel, we use the following further notation: $p_1=(1, 0, 0)$, $p_2=(0, 1, 0)$, $p_3=(0, 0, 1)$, $p_4=(1, 1, 0)$, $p_5=(0, 1, 1)$, $p_6=(1, 0, 1)$, $p_7=\bar p=(1, 1, 1)$.
{\it Step 1 (points $(i)-(iv)$).} Let $\rho=(\rho_{p_0}, \rho_{p_1}, \rho_{p_2}, \rho_{p_3}, \rho_{p_4}, \rho_{p_5}, \rho_{p_6}, \rho_{p_7})\in C_{\varepsilon}$ be fixed. Consider the finite set
$$
\tilde P_{p_0}=\{(p_1, \tau_1), (p_2, \tau_2), (p_3, \tau_3), (p_4, \tau_4), (p_5, \tau_5), (p_6, \tau_6), (p_7, \tau_7)\},
$$
whose elements are the couples composed by all the possible optimal admissible nodes $p_1,\ldots, p_7$ (starting from $p_0=(0, 0, 0)$), and the possible optimal switching instants $\tau_1,\ldots, \tau_7$, as derived in point $(ii)$, that is, for example, $\tau_2$ is the optimal switching instant in order to switch to $p_2=(0,1,0)$ with decision at $t=0$ in $p_0=(0,0,0)$ (independently whether the choice of $p_2$ is optimal or not).

For point $(iii)$, we argument as follows. At first observe that, at point $(ii)$, the multiplicity of the variables $P$ lies on the admissible subsequent node, but may also lie on the optimal switching instant (for a fixed node), if $\tau^-<\tau^+$, as in Remark \ref{unicoarrivo}. In order to make the solution $\rho$ consistent with the partition ${\mathcal P}_\varepsilon$, and to overcome the possible difficulties of the multivalued feature in time (making it at most discrete), we approximate the possible optimal switching instants $\tau_1,\ldots,\tau_7$ with the nodes of $\cP_{\varepsilon}$. In particular, for a generic switching instant $\tau_i$, we set
\begin{align*}
&\underline m(\tau_i, \varepsilon):=\max\{n\in\mathbb N:n\varepsilon\leq\tau_i\},\\
&\underline m(\tau_i,\varepsilon)\varepsilon=\text{the largest node not larger than $\tau_i$},\\
&\overline m(\tau_i, \varepsilon):=\min\{n\in\mathbb N:n\varepsilon\geq \tau_i\},\\
&\overline m(\tau_i,\varepsilon)\varepsilon=\text{the smallest node not smaller than $\tau_i$}.
\end{align*}
Then, if in the switching from $p$ to $p'$, the optimal switching instant $\tau_i$ belongs to the interval $[\underline m(\tau_i, \varepsilon)\varepsilon, \overline m(\tau_i, \varepsilon)\varepsilon]$, we select
\begin{equation}
\label{definizioneF}
\tilde\tau_{i,\ep}\in F(\tau_i)=\begin{cases}
\{\underline m(\tau_i, \ep)\ep\},&\tau_i\in[\underline m(\tau_i, \varepsilon)\ep, \underline m(\tau_i, \varepsilon)\ep+\frac{\ep}{2}[\\
\{\underline m(\tau_i, \varepsilon)\ep, \overline m(\tau_i, \varepsilon)\ep\},&\tau_i=\underline m(\tau_i, \varepsilon)\ep+\frac{\ep}{2}\\
\{\overline m(\tau_i, \varepsilon)\ep\},&\tau_i\in]\underline m(\tau_i, \varepsilon)\ep+\frac{\ep}{2}, \overline m(\tau_i, \varepsilon)\ep]
\end{cases}.
\end{equation}
In this way, the approximated variables $P_{\ep}$ in $(iii)$ replace every optimal pair $(p_i,\tau_i)\in P\subseteq\tilde P_{p_o}$ by the pairs (which we call $\varepsilon$-optimal) $(p_i,\tilde\tau_{i, \ep})$, $\tilde\tau_{i, \ep}\in F(\tau_i)$. Therefore, we construct all the possible $\varepsilon$-optimal switching paths $\pi$ with decision and switching times given by those approximated $\tilde\tau_{i, \ep}$, just taking, switch by switch, one and only one of the pairs above. For example, if $p_0\to p_1\to p_4\to p_7$ is an optimal path with $\tau_1,\tau_4,\tau_7$ the corresponding optimal switching instants, that is
$$
(p_1,\tau_1)\in P(p_0,0),\ (p_4,\tau_4)\in P(p_1,\tau_1),\
(p_7,\tau_7)\in P(p_4,\tau_4),
$$
then we consider all the possible $\varepsilon$-optimal paths $p_0\to p_1\to p_4\to p_7$ with $\varepsilon$-optimal switching instants $\tilde\tau_{j,\varepsilon}\in F(\tau_j)$, $j=1,4,7$, that is
$$
(p_1,\tilde\tau_{1,\varepsilon})\in P_\varepsilon(p_0,0),\ (p_4,\tilde\tau_{4,\varepsilon})\in P_\varepsilon(p_1,\tau_1),\
(p_7,\tilde\tau_{7,\varepsilon})\in P_\varepsilon(p_4,\tau_4),
$$
where
\begin{equation}
\label{eq:funzioniP_epsilon}
P_\varepsilon(p_i,s)=\{(p_j,F(\tau_j)):(p_j,\tau_j)\in P(p_i, s)\}.
\end{equation}
In particular, note that, if $\varphi(s)=[\tau_j^-,\tau_j^+]$ as in Remark \ref{unicoarrivo}, then $P_\varepsilon(p_i,s)$ contains all the pairs $(p_j,\tilde\tau_j)$ with $\tilde\tau_j=$ nodes of ${\mathcal P}_\varepsilon$ in $[\underline m(\tau_j^-,\varepsilon)\varepsilon,\overline m(\tau_j^+,\varepsilon)\varepsilon]$.

{\it Step 2 (points $(v)-(vii)$).} The aim is to build a multi-function $\rho\longmapsto\psi_{\ep}(\rho)\subset C_{\ep}$ with (compact and) convex images and closed graph, to which we will apply the fixed-point Kakutani-Ky Fan Theorem.

For each $\varepsilon$-optimal switching path $\pi$ of point $(iv)$, Step 1, we construct the corresponding evolution of the mass, assuming that all the agents (which here are assumed to be all at $p_0$ at time $t=0$) are following $\pi$. For example, for the possible $\varepsilon$-optimal path $p_0\to p_1\to p_4\to p_7$ as in Step 1, we would get
$$
\begin{array}{ll}
\displaystyle
\rho_0(t)=\begin{cases}
\rho^0,&0\le t<\tilde\tau_{1,\varepsilon}\\
0,&\tilde\tau_{1,\varepsilon}\le t\le T
\end{cases},\qquad
\rho_1(t)=\begin{cases}
0,&0\le t<\tilde\tau_{1,\varepsilon}\\
\rho^0,&\tilde\tau_{1,\varepsilon}\le t<\tilde\tau_{4,\varepsilon}\\
0,&\tilde\tau_{4,\varepsilon}\le t\le T
\end{cases},\\
\rho_4(t)=\begin{cases}
0,&0\le t<\tilde\tau_{4,\varepsilon}\\
\rho^0,&\tilde\tau_{4,\varepsilon}\le t<\tilde\tau_{7,\varepsilon}\\
0,&\tilde\tau_{7,\varepsilon}\le t\le T
\end{cases},\qquad
\rho_7(t)=\begin{cases}
0,&0\le t<\tilde\tau_{7,\varepsilon}\\
\rho^0,&\tilde\tau_{7,\varepsilon}\le t\le T
\end{cases},\\
\rho_i\equiv 0, \ i=2,3,5,6,
\end{array}
$$
and note that, by juxtaposition, $\rho^{\pi,\varepsilon}=(\rho_0,\rho_1,\ldots,\rho_7)\in C_\varepsilon$. Formally, as explained in \S4, such an evolution $\rho^{\pi,\varepsilon}$ can be seen as the constant interpolation of a decision-making solution $\rho^{\text{dm}}$ of \eqref{sistema2}, with coefficients $\lambda$ (to be understood associated to $\pi$, $\varepsilon$ and hence to the corresponding selection in $P_\varepsilon$) satisfying
$$
\l^{\pi,\varepsilon}_{i,j}(s, t)=\begin{cases}1,&(p_j, t)\in P_{\ep}(p_i, s)\cap\pi\\
0,&\text{otherwise}
\end{cases}.
$$
The aim is to construct $\psi_\varepsilon(\rho)$ as a suitable convexification of all those ``extremal'' evolutions $\rho^{\pi,\varepsilon}$. Such a convexification is constructed by taking into account the decision-making nodes $(p_0,0)$ and $(p_j,\tilde\tau_{j,\varepsilon})$. Still considering an example, suppose that the following paths (nodes $p_i$ and switching time $\tilde\tau_i$) are $\varepsilon$-optimal
$$
\begin{array}{ll}
\displaystyle
\pi^1: (p_0,0)\to(p_1,\tilde\tau_1^1)\to(p_4,\tilde\tau_4^1)\to(p_7,\tilde\tau_7^1),\\
\displaystyle
\pi^2: (p_0,0)\to(p_1,\tilde\tau_1^2)\to(p_6,\tilde\tau_6^2)\to(p_7,\tilde\tau_7^2),\\
\displaystyle
\pi^3: (p_0,0)\to(p_3,\tilde\tau_3^3)\to(p_6,\tilde\tau_6^3)\to(p_7,\tilde\tau_7^3),
\end{array}
$$
where we suppose
$$
0<\tilde\tau_1^1=\tilde\tau_1^2<\tilde\tau_6^2<\tilde\tau_3^3<\tilde\tau_6^3<\tilde\tau_4^1<\tilde\tau_7^1=\tilde\tau_7^2=\tilde\tau_7^3=T.
$$
We have a first decisional split in $p_0$ at $t=0$ between agents switching to $p_1$ and to $p_3$, respectively. We then have the convex coefficients $\lambda_{0,1}(0),\lambda_{0,3}(0)\in[0,1]$ with sum equal to $1$. Then another decisional split occurs in $p_1$ at $\tilde\tau_1=\tilde\tau_1^1=\tilde\tau_1^2$, giving the convex coefficients $\lambda_{1,4}(\tilde\tau_1)$, $\lambda_{1,6}(\tilde\tau_1)$, and no other decisional split occurs. We then get the evolutions
\begin{equation}
\label{convexification}
\begin{array}{ll}
\displaystyle
\rho_0(t)=\begin{cases}
\rho^0,&0\le t<\tilde\tau_1^1\\
\lambda_{0,3}(0)\rho^0,&\tilde\tau_1^1\le t<\tilde\tau_3^3\\
0,&\tilde\tau_3^3\le t\le T
\end{cases},\ \
\rho_1(t)=\begin{cases}
0,&0\le t<\tilde\tau_1^1\\
\lambda_{0,1}(0)\rho^0,&\tilde\tau_1^1\le t<\tilde\tau_6^2\\
\lambda_{1,4}(\tilde\tau_1)\lambda_{0,1}\rho^0,&\tilde\tau_2^6\le t<\tilde\tau_4^1\\
0,&\tilde\tau_4^1\le t\le T
\end{cases},\\
\rho_3(t)=\begin{cases}
0,&0\le t<\tilde\tau_3^3\\
\lambda_{0,3}(0)\rho^0,&\tilde\tau_3^3\le t<\tilde\tau_6^3\\
0,&\tilde\tau_6^3\le t\le T
\end{cases},\\
\rho_4(t)=\begin{cases}
0,&0\le t<\tilde\tau_4^1\\
\lambda_{1,4}(\tilde\tau_1)\lambda_{0,1}(0)\rho^0,&\tilde\tau_4^1\le t< T\\
0,&t=T
\end{cases},\\
\rho_6(t)=\begin{cases}
0,&0\le t<\tilde\tau_6^2\\
\lambda_{1,6}(\tilde\tau_1)\lambda_{0,1}(0)\rho^0,&\tilde\tau_6^2\le t<\tilde\tau_6^3\\
(\lambda_{1,6}(\tilde\tau_1)\lambda_{0,1}(0)+\lambda_{0,3}(0))\rho^0,&\tilde\tau_6^3\le t<T\\
0,&t=T
\end{cases},\\
\rho_7(t)=\begin{cases}
0,&0\le t<T\\
\rho^0,&t=T
\end{cases},\qquad\rho_2=\rho_5\equiv0.
\end{array}
\end{equation}
Again, by juxtaposition, we get an element of $C_\varepsilon$. The set $\psi_\varepsilon(\rho)\subseteq C_\varepsilon$ is then constructed by all the possible convexifications as above of all sets of extremal evolutions $\rho^{\pi,\varepsilon}$. See also the discussion after Definition \ref{def:epsilon_equilibrium} and the end of Remark \ref{generalsituation}. See Figure \ref{figure3} for a graphic representation of $\rho_6^{\text{dm}}$ and its constant interpolation $\rho_6$.

\begin{remark}
The functions $\lambda_{i, j}$ and their products as shown in the example above, together with the decisional and switching instants, give the coefficients $\lambda_{i, j}$ in the formal equations \eqref{sistema2}, for the decision-making part $\rho^{\text{dm}}$ of the evolution.
\end{remark}
\begin{figure}
\centering
\resizebox{!}{0.47\textwidth}{
\begin{tikzpicture}
\draw (-3.5,0)--(11,0)--(11,7)--(-3.5,7)--cycle;

\draw[-stealth] (0.7,2)--(5,2);
\draw[-stealth] (1,1.7)--(1,6);
\filldraw (1,2) circle (\rad) node [anchor=north east]{$0$};
\filldraw (4,2) circle (\rad) node [anchor=north]{$T$};
\draw[thick] (2,1.9)--(2,2.1);
\draw[thick] (3,1.9)--(3,2.1);
\draw[thick] (4,1.9)--(4,2.1);
\node at (2,1.7){$\tilde\tau_6^2$};
\node at (3,1.7){$\tilde\tau_6^3$};
\node at (5,1.7){$t$};
\node at (0.65,5.9){$\rho_6^{\text{dm}}$};

\draw[thick] (0.9,4)--(1.1,4);
\draw[thick] (0.9,5)--(1.1,5);
\node at (-0.4,4){$\lambda_{1,6}(\tilde\tau_1)\lambda_{0,1}(0)\rho^0$};
\node at (-1.3,5){$(\lambda_{1,6}(\tilde\tau_1)\lambda_{0,1}(0)+\lambda_{0,3}(0))\rho^0$};

\filldraw (2,4) circle (\rad);
\filldraw (3,5) circle (\rad);
\draw (2,2)--(2,4);
\draw (3,2)--(3,5);

\draw[-stealth] (5.7,2)--(10,2);
\draw[-stealth] (6,1.7)--(6,6);
\filldraw (6,2) circle (\rad) node [anchor=north east]{$0$};
\filldraw (9,2) circle (\rad) node [anchor=north]{$T$};
\draw[thick] (7,1.9)--(7,2.1);
\draw[thick] (8,1.9)--(8,2.1);
\draw[thick] (9,1.9)--(9,2.1);
\node at (7,1.7){$\tilde\tau_6^2$};
\node at (8,1.7){$\tilde\tau_6^3$};
\node at (10,1.7){$t$};
\node at (5.75,5.9){$\rho_6$};

\draw (7,2) circle (\rad);
\filldraw (7,4) circle (\rad);
\filldraw (8,5) circle (\rad);
\draw (8,4) circle (\rad);
\draw (9,5) circle (\rad);
\draw (7,4)--(8,4);
\draw (8,5)--(9,5);
\draw[thick] (6,2)--(7,2);

\end{tikzpicture}}
\caption{Representation of $\rho_6^{\text{dm}}$ and of its constant interpolation $\rho_6$.}\label{figure3}
\end{figure}
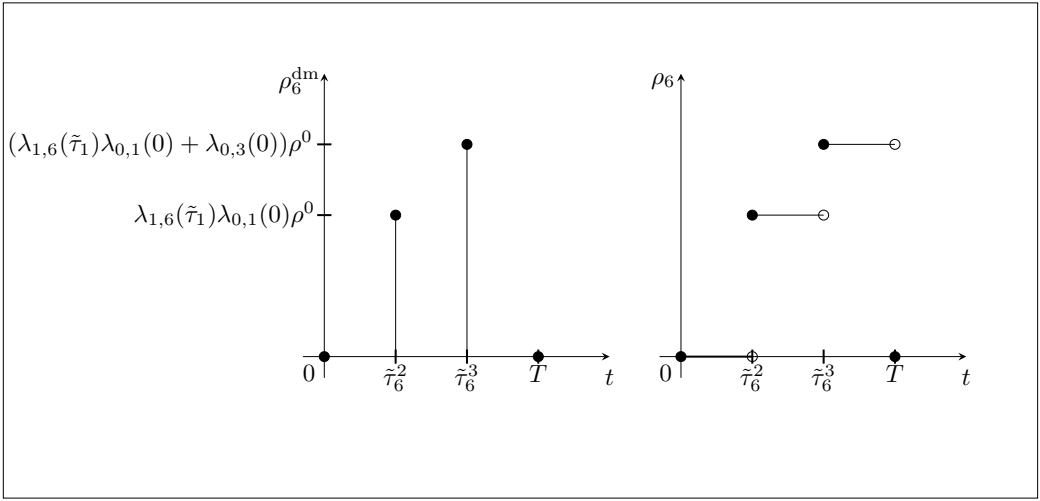

\begin{lemma}[point $(viii)$]
\label{fixedpointe}
For any $\rho\in C_{\varepsilon}$, the set $\psi_{\ep}(\rho)$ is a non-empty convex (and compact) subset of $C_{\varepsilon}$. Moreover, the map $\rho\longmapsto\psi_{\ep}(\rho)$ has closed graph.
\end{lemma}
\begin{proof}
Clearly the set $\psi_{\ep}(\rho)$ is non-empty and moreover it is convex. Indeed, if $\rho^1,\rho^2\in\psi_\varepsilon(\rho)$ and $\lambda\in[0,1]$, then $\lambda\rho^1+(1-\lambda)\rho^2\in\psi_\varepsilon(\rho)$. First, note that the extremal evolutions are a finite quantity $\{\rho^{\pi_{1,\varepsilon}},\ldots,\rho^{\pi_{r,\varepsilon}}\}$, because the number of $\varepsilon$-optimal paths, $\pi^{k,\varepsilon}$, $k=1,\ldots,r$, is finite. Hence we can consider both $\rho^1$ and $\rho^2$ as a convex combination, decisional node by decisional node (as described in Step 2), of all extremal evolutions, with convex coefficients sets $\Lambda^1$ and $\Lambda^2$ (note that the decisional nodes $(p_i,\tilde\tau_i)$ are determined by the fixed $\rho\in C_\varepsilon$ via \eqref{funzioniP}, \eqref{eq:funzioniP_epsilon}). This gives that $\lambda\rho^1+(1-\lambda)\rho^2$ is a same kind of convex combination of the extremal evolutions with set of convex coefficients $\lambda\Lambda^1+(1-\lambda)\Lambda^2$ (sum performed $\varepsilon$-optimal path by $\varepsilon$-optimal path, $\pi^{k,\varepsilon}$,  and decisional node by decisional node), and hence it belongs to $\psi_\varepsilon(\rho)$, which turns out to be convex.

Now, we prove that the multifunction $\rho\longmapsto\psi_{\ep}(\rho)$ has closed graph. From this, we also get the closedness of $\psi_{\ep}(\rho)$ and, since $C_{\varepsilon}$ is compact, it follows that $\psi_{\ep}(\rho)$ is compact too.

Consider a sequence $\{\rho^n\}_n\subset C_{\varepsilon}$ with $\rho^n\lra \rho$ in $C_{\varepsilon}$, that is $\rho\in C_\varepsilon$ and the convergence is in $L^2$. We want to show that for every $\rho'^n\in\psi_{\ep}(\rho^n)$ with $\rho'^n\lra\rho'$ in $C_{\varepsilon}$, we have $\rho'\in\psi_{\ep}(\rho)$.

Let us prove that, up to a subsequence, $\rho'^n\lra\tilde\rho'$ in $L^2$ with $\tilde\rho'\in C_{\ep}$ and $\tilde\rho'\in\psi_{\ep}(\rho)$. By the uniqueness of the limit in $L^2$, it must hold $\rho'=\tilde\rho'$, ending the proof. By Proposition \ref{lipcontV}, we have $V^n\lra V$ uniformly on $[0, T-k]$ for all $k>0$ (i.e., $V(p, \cdot, \rho^n)\lra V(p, \cdot, \rho)$ uniformly on $[0, T-k]$) and if $t'^n$ is optimal for $V(p, t^n, \rho^n)$ and $t^n\to t$, $t'^n\to t'$, then $t'$ is optimal for $V(p, t, \rho)$. In particular, note that, by Remark \ref{esistenzah}, since $t^n$ is a decisional instant for an optimal path starting at $t=0$, then $t^n\in[0,T-\bar h]$.  Therefore, denoting by $P^n, P_\varepsilon^n,P,P_\varepsilon$ the functions \eqref{funzioniP} and \eqref{eq:funzioniP_epsilon} corresponding to $\rho^n$ and $\rho$, respectively, we have
\begin{equation}
\label{convPep1}
(p'^n,t'^n)\in P^n(p,t^n)\ \text{and } (p'^n, t'^n)\to(p', t')\Rightarrow\ (p',t')\in P(p,t),
\end{equation}
and hence, by definition of $P_{\ep}$, \eqref{eq:funzioniP_epsilon} (see also the comment below it), in particular by the definition of $F$ in \eqref{definizioneF}, for every choice of $(p'^n, \tilde t'^n)\in P_{\ep}^n(p, t^n)$ there exists $(p', \tilde t')\in P_{\ep}(p, t)$ such that
\begin{equation}
\label{convPep2}
(p'^n, \tilde t'^n)\to (p', \tilde t')\quad\text{up to a subsequence (with $p'^n, t'^n, p', t'$ as in \eqref{convPep1})}.
\end{equation}
Moreover, since the nodes are finite, there exists $\bar n\in\mathbb N$ such that for every $p$,
\begin{equation}
\label{convergenzap}
p^n\to p\ \Rightarrow \ p^n=p\quad\text{for every }n\geq\bar n.
\end{equation}
Let $(\rho^{\pi_{1, \ep}},\ldots, \rho^{\pi_{r, \ep}})$ be the extremal points of $\psi_{\ep}(\rho)$, where $\pi_1,\ldots, \pi_r$ are the $\ep$-optimal paths. By \eqref{convergenzap}, we can assume that for $n$ sufficiently large, also in $\psi_{\ep}(\rho^n)$ the extremal points are exactly in the quantity $r$ and their sequences of nodes are the same as the ones of $\pi_1,\ldots, \pi_r$ and only the decisional and switching instants may change with $n$. Let us denote by $\rho^{\pi_1, n, \ep},\ldots, \rho^{\pi_r, n, \ep}$ those extremal points. Then, for $n$ sufficiently large, $\rho'^n\in\psi_{\ep}(\rho^n)$ is a convex combination, constructed as in Step 2, of the extremal points $\rho^{\pi_1, n, \ep},\ldots, \rho^{\pi_r, n, \ep}$. Let $\l_{i, j}^n(\tilde t^n)\in[0, 1]$ be the corresponding coefficients for the generic decisional instant $\tilde t^n$. Up to a subsequence, we can assume that $\tilde t^n\to\tilde t$ and $\l_{i, j}^n(\tilde t^n)\to\l_{i, j}=:\l_{i, j}(\tilde t)\in[0, 1]$ and also $\tilde t'^n\to\tilde t'$ with $(p', \tilde t'^n)\in P_{\ep}^n(p, \tilde t^n)$ and, by \eqref{convPep2}, $(p', \tilde t')\in P_{\ep}(p, \tilde t)$. Since $\tilde t'^n, \tilde t^n$ assume only discrete values on partition $\cP_{\ep}$, we can also assume $\tilde t'^n=\tilde t'$ and $\tilde t^n=\tilde t$ for $n$ sufficiently large. Hence the extremal points $\rho^{\pi_1, n, \ep},\ldots, \rho^{\pi_r, n, \ep}$ are exactly the same as the ones of the limit case $\psi_{\ep}(\rho)$: the same $\ep$-optimal paths $\pi_1,\ldots, \pi_r$ with the same decisional and switching instants. The only convergence is in the convex coefficients.

Now, we construct $\tilde\rho'$ as the convex combination of the extremal points with limit coefficients $\l_{i, j}$. Obviously $\tilde\rho'\in C_{\ep}$ and $\tilde\rho'^n\lra\tilde\rho'$ in $L^2$. To conclude, we have to prove that $\tilde\rho'\in \psi_{\ep}(\rho)$. In particular, we have to show that if $(p_j, \tilde t')\notin P_{\ep}(p_i, \tilde t)$, then the corresponding $\l_{i, j}(\tilde t)=0$. This is true because, if $\l_{i, j}(\tilde t)$ was greater than $0$, then $\l^n_{i, j}(\tilde t^n)>0$ by convergence and hence $(p_j, \tilde t'^n)\in P_{\ep}^n(p_i, \tilde t^n)$, and this is in contradiction with \eqref{convPep2}. Therefore $\tilde\rho'\in\psi_{\ep}(\rho)$ and we conclude because, by construction, $\rho'^n\lra\tilde\rho'$ in $L^2$ since the convergence of the coefficients $\l_{i, j}^n$ gives the convergence of the constant values of $\rho'^n$ on the partition $\cP_{\ep}$ to the constant values of $\tilde\rho'$.
%
\end{proof}
\begin{remark}
\label{generalsituation}
Observe that the general case $N>3$ works with the same ideas and tools, being careful that we will have a more complex network (i.e, many more nodes and paths, that is a more complex topology of the network), which makes the fixed-point procedure above certainly harder from a computational point-of-view but even just from a notational one, already for what concerns the analytical description of $\psi_{\ep}$ (see for example the description of $\rho_6$ in the simple case in \eqref{convexification}). Moreover, here above, for simplicity, we considered only paths starting from $p_0=(0, 0, 0)$ and that, at the initial time $t=0$, all the agents are in $p_0$, that is $\rho_i(0)=0$ for all $i\neq0$. The case where at the initial time the mass is possible distributed to different nodes, up to suitably construct the evolutions as in Step 2, which will be more knotty, does not change the proof too much (we may have more involved intersections and overlaps of switches, still in a finite number, as $\rho_6$ in \eqref{convexification} but probably in a more complicated way).

Still considering the network in Figure \ref{figure2} as in \eqref{convexification}, with the same enumeration of nodes $p_0,p_1,\ldots,p_7=\bar p$, in order to give an idea of the descriptive and notational complexity of the construction of $\psi_{\ep}$, already in the case of that simple network, but with a generic initial distribution $\rho^0=(\rho^0_0,\rho^0_1,\ldots\rho^0_7)$, if we consider, for instance, the flow $\rho_6$ through the node $p_6$, we have
$$
\rho_6=\rho^{6, 6}+\rho^{1, 6}+\rho^{3, 6}+\rho^{0,1,6}+\rho^{0,3,6}.
$$
The term $\rho^{6, 6}$ corresponds to the flow of the agents that at time $t=0$ are already in $p_6$: all of them, at the decisional instant $t=0$, choose a switching instant $\tau_{6,7}$ optimally generated as in \eqref{definizioneF} in order to switch from $p_6$ to $p_7$.

The term $\rho^{1, 6}$ corresponds to the flow, through $p_6$, of the agents that at $t=0$ were in $p_1$: all of them, at the decisional instant $t=0$, choose a switching instant $\tau_{1,6}$ optimally generated as in \eqref{definizioneF} in order to switch from $p_1$ to $p_6$, together with the corresponding fraction $\l_{1,6}$ of agents performing such a switch. Hence, at the instant $\tau_{1,6}$, the mass of agents $\l_{1,6}\rho^0_1$ switches from $p_1$ to $p_6$. Such a mass of agents, at the (decisional) instant $\tau_{1,6}$, optimally chooses a switching instant $\tau_{1,6,7}$ in order to switch from $p_6$ to $p_7$.

The term $\rho^{3,6}$ is constructed similarly to $\rho^{1,6}$ by replacing $p_1$ with $p_3$.

The term $\rho^{0,1,6}$ corresponds to the flow, through $p_6$, of the agents that at $t=0$ were in $p_0$: all of them, at the decisional instant $t=0$, choose a switching instant $\tau_{0,1}$ optimally generated as in \eqref{definizioneF} in order to switch from $p_0$ to $p_1$, together with the corresponding fraction $\l_{0,1}$ of agents performing such a switch. Hence, at the instant $\tau_{0,1}$, the mass of agents $\l_{0,1}\rho^0_0$ switches from $p_0$ to $p_1$. Such a mass of agents, at the (decisional) instant $\tau_{0,1}$, optimally chooses a switching instant $\tau_{0,1,6}$ in order to switch from $p_1$ to $p_6$, together with the fraction $\l_{0,1,6}$ of agents performing such a switch. Therefore, at the instant $\tau_{0,1,6}$, the mass of agents $\l_{0,1,6}\l_{0,1}\rho^0_0$, switches from $p_1$ to $p_6$. Such a mass of agents, at the (decisional) instant $\tau_{0,1,6}$, optimally chooses a switching instant $\tau_{0,1,6,7}$ in order to switch from $p_6$ to $p_7$.

The term $\rho^{0,3,6}$ is constructed similarly to $\rho^{0,1,6}$ by replacing $p_1$ with $p_3$.

Obviously, the coefficients $\l$ above must be constrained to have sum equal to $1$ with the other corresponding coefficients. For instance, $\l_{0,1,6}+\l_{0,1,4}=1$. Finally note that in the simple case \eqref{convexification}, $\rho_6$ corresponds to $\rho^{0,1,6}+\rho^{0,3,6}$ only, and, in particular, $\l_{3,6}=1$, which means that $\l_{3,5}=0$, for the optimality hypotheses assumed in that example.

As said before, we end this remark with a possible general definition of $\psi_\varepsilon$.  Given $\rho\in C_\varepsilon$, another element $\tilde\rho\in C_\varepsilon$ belongs to $\psi_\varepsilon(\rho)$ if and only if the following holds. We give the list of indices and notation we are going to use.
\begin{itemize}
\item[1)] $j$: Index of the node of the network, $0\le j\le 2^N-1$;
\item[2)] $\ell_j$: Index of the optimal path starting from the node $p_j$, $1\le \ell_j\le\mu_j$, for some $\mu_j\in\mathbb{N}\setminus\{0,\}$;
\item[3)] $\pi_{j,\ell_j}$: $\ell_j$-th optimal path starting from $p_j$;
\item[4)] $\nu_{j,\ell_j}$: Ordering number of the switch in the path $\pi_{j,\ell_j}$, $1\le\nu_{j,\ell_j}\le\xi_{j,\ell_j}$, for some $\xi_{j,\ell_j}\in\mathbb{N}\setminus\{0\}$;
\item[5)] $\left(p_{j,\ell_j,\nu_{j,\ell_j}}, t_{j,\ell_j,\nu_{j,\ell_j}}\right)$: $\nu_{j,\ell_j}$-th pair (node/time) of $\pi_{j,\ell_j}$;
\item[6)] $\kappa_{j,\ell_j,\nu_{j,\ell_j}}$: Number of nodes in ${\cal I}_{p_{j,\ell_j,\nu_{j,\ell_j}}}$ such that there exists $t>t_{j,\ell_j,\nu_{j,\ell_j}}$ such that, being in $p_{j,\ell_j,\nu_{j,\ell_j}}$ at the time $t_{j,\ell_j,\nu_{j,\ell_j}}$, it is optimal to decide to switch on that node at the time $t$;
\item[7)] For all $(j,\ell_j,\nu_{j,\ell_j})$ and for all $1\le\kappa\le\kappa_{j,\ell_j,\nu_{j,\ell_j}}$ we take $\lambda_{j,\ell_j,\nu_{j,\ell_j},k}\in[0,1]$, such that
$$
\sum_{\kappa=1}^{\kappa_{j,\ell_j,\nu_{j,\ell_j}}}\lambda_{j,\ell_j,\nu_{j,\ell_j},\kappa}=1,
$$
and we denote by $p_{j,\ell_j,\nu_{j,\ell_j},\kappa}$ the corresponding node in ${\cal I}_{p_{j,\ell_j,\nu_{j,\ell_j}}}$;
\item[8)] For $0\le\nu\le\nu_{j,\ell_j}-1$ let $i_\nu$ such that $p_{j,\ell_j,\nu,i_\nu}=p_{j,\ell_j,\nu+1}$.
\end{itemize}

The statement is then: $\tilde\rho\in\psi_\varepsilon(\rho)$ if and only if we can choose the numbers $\lambda_{j,\ell_j,\nu_{j,\ell_j},\kappa}$ as in 7) such that, for all $0\le k\le 2^N-1$ and $t\in[0,T[$ we have
$$
\tilde\rho_k(t)=\sum_{(j,\ell_j,\nu_{j,\ell_j})\Big|p_{j,\ell_j,\nu_{j,\ell_j}}=p_k,\ t_{j,\ell_j,\nu_{j,\ell_j}}\le t<t_{j,\ell_j,\nu_{j,\ell_j}+1}}\left(\left(\prod_{\nu=0}^{\nu_{j,\ell_j}-1}
\lambda_{j,\ell_j,\nu,i_\nu}\right)\rho_j^0\right).
$$
\end{remark}

\begin{theorem}
\label{maintheorem}
Under all the hypotheses stated in \S2, there exists an $\ep$-mean-field equilibrium of system \eqref{mfgsys} (see Definition \ref{def:epsilon_equilibrium}).
\end{theorem}
\begin{proof}
The proof follows from Lemma \ref{fixedpointe}, Remark \ref{generalsituation} and the fixed-point Kakutani-Ky Fan Theorem.
\end{proof}

\section{On the limit $\varepsilon\to0$ and the existence and uniqueness of a mean-field equilibrium}
\label{casolimite}
In the sequel, we denote by $\rho_{\ep}$ a fixed point for $\psi_{\ep}(\rho)$, i.e., a total mass satisfying $\rho_{\ep}\in\psi_{\ep}(\rho_{\ep})$. The existence of such fixed points is proved in the previous section and now we will perform the limit procedure as $\ep\to0$, obtaining as limit $\rho\in L^2([0,T], [0, 1])^{2^N}$ such that $\rho\in\psi(\rho)$, where $\psi$ is constructed as in the previous points $(i)$--$(viii)$ with the only difference that we do not perform the approximation $P_\varepsilon$ in $(iii)$, but we just consider the function $P$, (\ref{funzioniP}). Hence $\rho$, together with its convexity coefficients, will be a solution of \eqref{mfgsys} and a mean-field equilibrium. Recall that, see Remark \ref{esistenzah}, in the non-approximated case (i.e., no $\varepsilon$-partition) the mass flow of the agents is still piecewise constant with pieces of length at least $\bar h>0$, but not necessarily based on the $\varepsilon$-partition. Hence, the construction of the multifunction $\psi$ is similar to the one of $\psi_\varepsilon$. The only difference is that in the construction of $\psi_\varepsilon$, in the switching optimization procedure, we first use the function $P$ (\ref{funzioniP}) and then adjust the switching time by $P_\varepsilon$ as described in Subsection \ref{esistenzamfgapp}, whereas, in the general non-approximated case, we just take the switching instant given by the function $P$.
\begin{definition}
\label{def:equilibrium}
A mean field equilibrium for (\ref{mfgsys}), in the case of no $\varepsilon$-partition, is a function $\rho$ such that $\rho\in\psi(\rho)$.
\end{definition}
One of the main problems in performing such a limit is the fact that the functions $t\longmapsto\tau=\varphi(t)$ (see Remark \ref{unicoarrivo}) may be multivalued, and, in particular, with a continuum (an interval, see Remark \ref{unicoarrivo}) as image of $t$. This problem was bypassed in the previous section using the time-discretizetion given by the partition $\mathcal{P}_\varepsilon$. We first assume that the functions $\varphi$ are not multivalued and we prove, in such a case, the existence of a mean-field equilibrium, that is of a function $\rho\in L^2$ such that $\rho\in\psi(\rho)$.
\begin{theorem}
\label{casoepsilon}
Under all the hypotheses stated in \S2 and assuming the single-valued feature of $\varphi$, there exists a mean-field equilibrium of system \eqref{mfgsys}, that is there exists $\rho\in L^2$ such that $\rho\in\psi(\rho)$ (see Definition \ref{def:equilibrium}).
\end{theorem}
\begin{proof}
First of all note that, fixed $\rho$, under the hypothesis on $\varphi$, for every decisional instant $t$ and node $p_i$, there exists a unique optimal switching instant $\tau$ for the switch to $p_j$, that is $(p_j,\tau)\in P(t,p_i)$. This fact gives that the mass evolution $\rho'\in\psi(\rho)$ is also piecewise constant and similarly constructed as in Step 2, \S\ref{esistenzamfgapp}, with the only difference that now the pieces of constancy are not fixed a priori (we do not have the partition $\mathcal{P}_\varepsilon$). Moreover, for all $\varepsilon>0$, the function $P_\varepsilon$, \eqref{eq:funzioniP_epsilon}, evaluated at $(t,p_i)$, generates at most two $\varepsilon$-approximated switching instants for the switch to $p_j$: the possible approximation $\tilde\tau_\varepsilon$ of $\tau$ by the function $F$ in \eqref{definizioneF} (and not the whole intersection of the nodes of the partition with the interval $\varphi(t)$ in the case of multivalued feature). Finally, $\tilde\tau_\varepsilon\to\tau$ as $\varepsilon\to0$. Note that, see Definition \ref{optimallygen}, the function $\varphi$, the optimally generated switching time, depends on $p_i$ and $p_j$. Here for simplicity we do not display such a dependence. The proof is made by a generic pair $(p_i, p_j)$ with $p'\in{\cal I}_{p_i}$, and of course, being such a pairs in a finite quantity, we have the uniformity of the convergence, with respect to the pair.

Now, recall that (see the beginning of \S\ref{esistenzamfgapp}) the fixed points $\rho_\varepsilon$ are piecewise constant with at most a fixed number $M$ of pieces of constancy. Hence, possibly extracting a subsequence, we can make such intervals of constancy converge as well as the corresponding values of the constants. We then get a function $\rho$ such that, up to a subsequence, $\rho_\varepsilon\to\rho$ in $L^2$. The convergence of the constant values is obviously constructed by the convergence, up to a subsequence, of the convex coefficients $\lambda_{i, j}^\varepsilon\in\mathbb{R}$ evaluated on the decisional instants and implemented at the corresponding $\varepsilon$-approximated instants as in Step 2, \S\ref{esistenzamfgapp}. Note that the decisional and switching instants are the extremal points of the intervals of constancy, and also that, being the number of possible cases finite, we can assume, up to a subsequence, that those ones are decisional and switching instants for the same switch from $p_i$ to $p_j$, i.e. for the same $i$ and $j$ for all $\varepsilon$. Finally note that $\rho_\varepsilon$, being a fixed point of $\psi_\varepsilon$, is exactly constructed by its coefficients $\lambda_{i, j}^\varepsilon$ implemented on the nodes that are generated by $\rho_\varepsilon$ itself via $P_\varepsilon$.

Arguing as in the proof of Lemma \ref{fixedpointe}, using Proposition \ref{lipcontV} and similar convergence for $\varepsilon\to0$ as in \eqref{convPep1} and \eqref{convPep2}, we get that $\rho\in\psi(\rho)$ (i.e.: $\rho$ is constructed by the coefficients $\lambda_{i, j}$ implemented on the nodes that are generated by $\rho$ itself via $P$, and moreover if the switch is not optimal, then $\lambda_{i, j}=0$).
\end{proof}
In Appendix \ref{appendixB}, we give an explicit example of possible costs that guarantee the single-valued feature of $\phi$.
\subsection{The general case: $\varphi$ multivalued}
\label{generalcase}
Without the single-valued hypothesis on $\varphi$, the passage to the limit as $\ep\to0$ is more involved. Indeed, if the image of the decisional time $t$ is an interval $[\tau^-, \tau^+]$, in the $\ep$-approximation case we discretize it through the partition $\cP_{\ep}$ and, on every node, we get a value $\l_{i, j}^{\ep}(t, \cdot)$ which composes with the others. Formally, we have a sum of weighted delta functions on the nodes of $\cP_{\ep}$ inside $[\tau^-, \tau^+]$. In the the limit as $\ep\to0$, we obtain instead a possible sum of functions $\l_{i, j}(t, \cdot)$, defined on the whole interval $[\tau^-, \tau^+]$ and other sums of delta functions. Hence the situation is more complex, including the interpretation of system \eqref{sistema2}. A rigorous investigation of this situation is going to be the subject of future works. Again considering a particular case, where $\rho_{\ep}\lra\rho$ in $L^2$ and $\rho$, via the functions $P$, generates functions $\phi$ not multivalued, then $\rho$ may be a mean-field equilibrium because the proof of Theorem \ref{casoepsilon} can be probably adapted. Also for this case the details have not been checked. However, in Appendix \ref{appendixB}, we give an explicit example of possible costs that guarantee the single-valued feature of $\phi$.
\subsection{On the uniqueness of the equilibrium}
As we said in the Introduction, the uniqueness of the equilibrium is often proved by assuming the Lasry-Lions monotonicity condition on the cost (see \cite{lasrylions}). Our problem does not immediately fit into such a property because of its deterministic and network-type features and the presence of two kinds of time-variables. Anyway, in Appendix \ref{appendix}, we try to show, by two simple examples, how a monotonicity-type condition can be promising in order to study the uniqueness of the equilibrium but the real implementation of that condition in our model is completely left to future studies.
\appendix
\section{On the uniqueness of the equilibrium}
\label{appendix}
We first recall that, as in \S\ref{casolimite}, a mean-field equilibrium is a function $\rho\in L^2$ such that $\rho\in\psi(\rho)$, which means that $\rho$ is a juxtaposed convex combination of the extremal evolutions generated by $\rho$ itself via the optimization functions $P$ \eqref{funzioniP}.

The examples we are going to show do not necessarily meet in all their aspects our model studied in previous sections. They are just inspiring examples about the possible use of the monotonicity property.
\begin{example}
\label{example1uni}
\begin{figure}[h]
\centering
\begin{tikzpicture}[very thick,decoration={
    markings,
    mark=at position 0.5 with {\arrow{>}}}
    ]
\draw[postaction={decorate}] (0,0)--(2,-2);
\draw[postaction={decorate}] (0,0)--(2,0);
\draw[postaction={decorate}] (0,0)--(2,2);
\draw[postaction={decorate}] (2,-2)--(6,0);
\draw[postaction={decorate}] (2,0)--(6,0);
\draw[postaction={decorate}] (2,2)--(6,0);
\filldraw[\colo] (0,0) circle (\rad)  node [anchor=east] {$p_0$};
\filldraw[\colo] (2,-2) circle (\rad)  node [anchor=north] {$p_3$};
\filldraw[\colo] (2,0) circle (\rad)  node [anchor=south] {$p_2$};
\filldraw[\colo] (2,2) circle (\rad)  node [anchor=south] {$p_1$};
\filldraw[\colo] (6,0) circle (\rad)  node [anchor=west] {$p_4$};
\end{tikzpicture}
\caption{The network of Example \ref{example1uni}.}\label{figure4}
\end{figure}
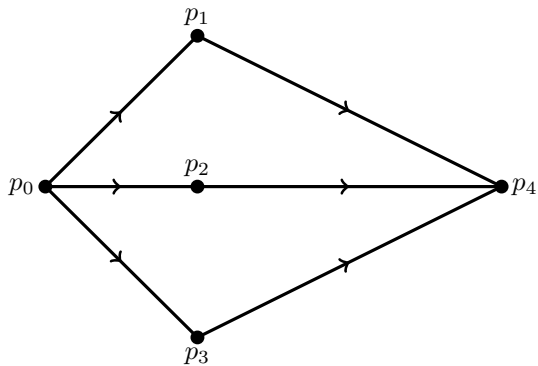
Consider the network in Figure \ref{figure4}, where the goal is to start from $p_0$ and to arrive to $p_4$, along the three possible paths: $p_0\to p_1\to p_4$, $p_0\to p_2\to p_4$ and $p_0\to p_3\to p_4$. Moreover, we suppose that all the agents at the time $t=0$ are in $p_0$, that at the time $t=1$ they are all forced to switch to one of the three nodes $p_1$, $p_2$ and $p_3$, and that at the time $t=T=2$ they are all forced to switch to $p_4$, ending the game. Since the switching instants are fixed and the significant nodes are just $p_1$, $p_2$ and $p_3$, we only give the cost of stay on such nodes respectively, independently of time: $C_1(\rho_1)=\rho_1, C_2(\rho_2)=2\rho_2, C_3(\rho_3)=3\rho_3$, where $\rho_i$ is the mass in the node $p_i$. In this case a mean-field equilibrium is given by $(\lambda_1, \lambda_2,\lambda_3)=(6/11,3/11,2/11)$, which means that, denoted by $\rho_0$ the initial distribution in $p_0$, at time $t=1$ the fraction $\lambda_i\rho_0$ switches to the node $p_i$, $i=1,2,3$. Indeed, with these fractions all the costs $C_1,C_2,C_3$ are equal to $(6/11)\rho_0$. Hence if all the agents in $p_0$ conjecture such a distribution, then all the possible generated extremal distributions are the following ones: $(\rho_0,0,0), (0,\rho_0,0), (0,0,\rho_0)$, that is all the switches are optimal. The actual mass $(\lambda_1\rho_0,\lambda_2\rho_0,\lambda_3\rho_0)$ is then a convex combination of the generated extremal distributions with convex coefficients $(\lambda_1,\lambda_2,\lambda_3)$, and hence it is a mean field equilibrium. By linearity of the costs, the coefficients $\lambda_i$ are easily calculated by imposing $C_1(\lambda_1)=C_2(\lambda_2)=C_3(\lambda_3)$ with the constraint $\lambda_i\in[0,1]$ and $\lambda_1+\lambda_2+\lambda_3=1$, and they are the only ones satisfying the system and the constraint. Note that if, for example, we are looking for a possible equilibrium using just the nodes $p_1$ and $p_2$, that is we look for $\lambda_1,\lambda_2\ge0$, $\lambda_1+\lambda_2=1$ and $C_1(\lambda_1)=C_2(\lambda_2)$, we find $\lambda_1=2/3, \lambda_2=1/3$ and then we have the distribution $(\lambda_1\rho_0,\lambda_2\rho_0,0)$. But such a distribution is not an equilibrium because it gives the costs $\left((2/3)\rho_0,(2/3)\rho_0,0\right)$, which generates the only extremal distribution $(0,0,\rho_0)$: all agents switch to $p_3$. And $(\lambda_1\rho_0,\lambda_2\rho_0,0)$ is not a convex combination of (i.e., is not equal to) the singleton $\{(0,0,\rho_0)\}$. The problem then has a unique equilibrium which is given by $((6/11)\rho_0,(3/11)\rho_0,(2/11)\rho_0)$.

Note that, whenever we find a triple of convex coefficients $(\lambda_1,\lambda_2,\lambda_3)$ such that $C_1(\lambda_1)=C_2(\lambda_2)=C_3(\lambda_3)$, then the corresponding distribution $(\lambda_1\rho_0,\lambda_2\rho_0,\lambda_3\rho_0)$ is an equilibrium because it gives the same costs along any path, and then generates all the extremal distributions $(\rho_0,0,0),(0,\rho_0,0),(0,0,\rho_0)$ of which it is a convex combination. The question about uniqueness is then: given three functions $C_i:[0,1]\to\mathbb{R}$, $i=1,2,3$, under which condition there exists at most one triple of convex coefficients $(\lambda_1,\lambda_2,\lambda_3)$ such that
\begin{equation}
\label{eq:equality}
C_1(\lambda_1\rho_0)=C_2(\lambda_2\rho_0)=C_3(\lambda_3\rho_0)\text{?}
\end{equation}
A condition that guarantees such a uniqueness is the following monotonicity property which is, in our discrete case, the condition in \cite{lasrylions}:
\begin{equation}
\label{eq:monotonicity}
\begin{array}{ll}
\displaystyle
\sum_{i=1}^3\left(C_i(\lambda'_i\rho_0)-C_i(\lambda_i''\rho_0)\right)(\lambda_i'-\lambda_i'')>0\\
\displaystyle
\text{for all }(\lambda'_1,\lambda'_2,\lambda'_3)\neq(\lambda''_1,\lambda''_2,\lambda''_3)\ \mbox{convex triples and for any }\rho_0>0.
\end{array}
\end{equation}
Indeed, let us suppose that there are two convex triples
$$
(\lambda_1',\lambda_2',\lambda_3')=(\lambda_1',\lambda_2',1-\lambda'_1-\lambda_2'),\ (\lambda_1'',\lambda_2'',\lambda_3'')=(\lambda_1'',\lambda_2'',1-\lambda''_1-\lambda_2'')
$$
satisfying \eqref{eq:equality}, and denoting by $C', C''$ the common costs, for the single triple respectively, we obtain
$$
\sum_{i=1}^3(C_i'-C_i')(\lambda_i'-\lambda_i'')=(C'-C'')\sum_{i=1}^2(\lambda'_i-\lambda''_i)+(C'-C'')(1-\lambda_1'-\lambda_2'-1+\lambda_1''+\lambda_2'')=0
$$
and hence, by \eqref{eq:monotonicity}, $(\lambda_1',\lambda_2',\lambda_3')=(\lambda_1'',\lambda_2'',\lambda_3'')$.
\end{example}
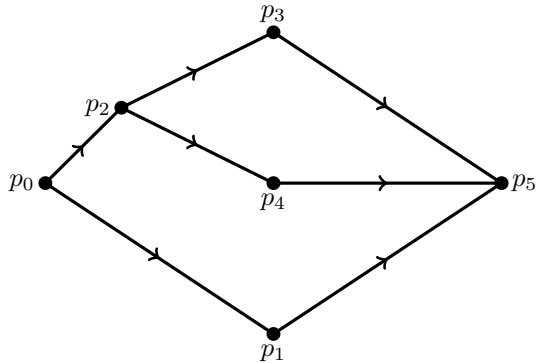
\begin{figure}[h]
\centering
\begin{tikzpicture}[very thick,decoration={
    markings,
    mark=at position 0.5 with {\arrow{>}}}
    ]
\draw[postaction={decorate}] (0,0)--(3,-2);
\draw[postaction={decorate}] (0,0)--(1,1);
\draw[postaction={decorate}] (3,-2)--(6,0);
\draw[postaction={decorate}] (3,2)--(6,0);
\draw[postaction={decorate}] (1,1)--(3,2);
\draw[postaction={decorate}] (1,1)--(3,0);
\draw[postaction={decorate}] (3,0)--(6,0);
\filldraw[\colo] (0,0) circle (\rad)  node [anchor=east] {$p_0$};
\filldraw[\colo] (3,0) circle (\rad)  node [anchor=north] {$p_4$};
\filldraw[\colo] (3,2) circle (\rad)  node [anchor=south] {$p_3$};
\filldraw[\colo] (3,-2) circle (\rad)  node [anchor=north] {$p_1$};
\filldraw[\colo] (1,1) circle (\rad)  node [anchor=east] {$p_2$};
\filldraw[\colo] (6,0) circle (\rad)  node [anchor=west] {$p_5$};
\end{tikzpicture}
\caption{The network of Example \ref{example2uni}.}\label{figure5}
\end{figure}
\begin{example}
\label{example2uni}
Consider the network in Figure \ref{figure5}. The goal is to start from $p_0$ and to reach $p_5$ among one of the possible paths $p_0\to p_1\to p_5$, $p_0\to p_2\to p_3\to p_5$ and $p_0\to p_2\to p_4\to p_5$. Again, the agents at $t=0$ are all in $p_0$, with distribution $\rho_0$, at time $t=1$ they are forced to switch to $p_1$ or to $p_2$, at time $t=3/2$ the agents in $p_2$ are forced to switch to $p_3$ or $p_4$ and at the time $t=T=2$ they are all forced to switch to $p_5$. The costs are $C_1(\rho_1)=\rho_1$, $C_2(\rho_2)=4\rho_2$, $C_3(\rho_3)=3\rho_3$, $C_4(\rho_4)=2\rho_4$. Moreover, the costs are also multiplied by the amount of the time spent on the node. We denote by $(\lambda_1,\lambda_2,\lambda_{2,3},\lambda_{2,4})$ the coefficients of a possible equilibrium, that is: at $t=1$ the fraction given by $\lambda_1\rho_0$ switches to $p_1$ and the fraction given by $\lambda_2\rho_0$ switches to $p_2$; at time $t=3/2$, the fraction $\lambda_2\lambda_{2,3}\rho_0$ switches from $p_2$ to $p_3$ and the fraction $\lambda_2\lambda_{2,4}\rho_0$ switches from $p_2$ to $p_4$. Still by linearity of the costs, such coefficients are founded by solving
\begin{equation}
\label{eq:system_costs}
\begin{cases}
2\lambda_2+\frac{3}{2}\lambda_2\lambda_{2,3}=\lambda_1\\
2\lambda_2+\lambda_2\lambda_{2,4}=\lambda_1\\
\lambda_1+\lambda_2=\lambda_{2,3}+\lambda_{2,4}=1
\end{cases},
\end{equation}
which corresponds to, taking also account of the time spent on the node,
\begin{equation}
\label{eq:system_costs_2}
\begin{cases}
\frac{C_2(\lambda_2\rho_0)}{2}+\frac{C_3(\lambda_2\lambda_{2,3}\rho_0)}{2}=C_1(\lambda_1\rho_0)\\
\frac{C_2(\lambda_2\rho_0)}{2}+\frac{C_4(\lambda_2\lambda_{2,4}\rho_0)}{2}=C_1(\lambda_1\rho_0)\\
\lambda_1+\lambda_2=\lambda_{2,3}+\lambda_{2,4}=1
\end{cases}.
\end{equation}
From \eqref{eq:system_costs}, we obtain the unique solution
$$
(\lambda_1,\lambda_2,\lambda_{2,3},\lambda_{2,4})=\left(\frac{13}{18},\frac{5}{18},\frac{2}{5},\frac{3}{5}\right).
$$
This is an equilibrium because it generates the distribution
\begin{equation}
\label{eq:distribution}
\left(\frac{13}{18}\rho_0,\frac{5}{18}\rho_0,\frac{1}{9}\rho_0,\frac{1}{6}\rho_0\right),
\end{equation}
which gives the cost, for each one of the three paths, equal to $13/18$. Hence all the paths are equivalent and the distribution generates all the possible extremal evolutions $(\rho_0,0,0,0)$, $(0,\rho_0,\rho_0,0)$, $(0,\rho_0,0,\rho_0)$ of which \eqref{eq:distribution} is a juxtaposed convex combination.

Similarly as in \eqref{eq:monotonicity}, the uniqueness of the solution of \eqref{eq:system_costs_2} is guaranteed by the following monotonicity conditions
\begin{equation}
\label{eq:meglio}
\begin{array}{ll}
\displaystyle
\left\{
\begin{array}{ll}
\displaystyle
\sum_{i=3}^4\left(C_i(\lambda\lambda'_{2,i}\rho_0)-C_i(\lambda\lambda_{2,i}''\rho_0)\right)(\lambda_{2,i}'-\lambda_{2,i}'')>0\quad\text{for every }\lambda>0,\\
\displaystyle
\end{array}
\right.\\
\displaystyle
\left\{
\begin{array}{ll}
\displaystyle
\left(C_1(\lambda'_1\rho_0)-C_1(\lambda''_1\rho_0)\right)(\lambda_1'-\lambda_1'')\\
\displaystyle
+\frac{1}{2}\left(C_2(\lambda'_2\rho_0)+C_3(\lambda_2'\lambda_{2,3}'\rho_0)-C_2(\lambda_2''\rho_0)-C_3(\lambda_2''\lambda_{2,3}''\rho_0)\right)(\lambda_2'-\lambda_2'')>0
\end{array}
\right.\\
\displaystyle
\ \ \forall(\lambda'_1,\lambda'_2)\neq(\lambda''_1,\lambda''_2),\ (\lambda_{2,3}',\lambda_{2,4}')\neq(\lambda_{2,3}'',\lambda_{2,4}'')\ \mbox{convex pairs and }\rho_0>0.
\end{array}
\end{equation}
Indeed, by the second inequality we have the uniqueness of the pair of convex coefficients $(\l_1, \l_2)$, which, putting $\l=\l_2$ in the first inequality, gives the uniqueness of the pair $(\lambda_{2, 3}, \lambda_{2, 4})$.
\end{example}
\begin{remark}
Similarly as in Example \ref{example1uni} (see Figure \ref{figure4}) when the number of the nodes is $n$ instead of $3$, the uniqueness of the $n$-string of convex coefficients satisfying $C_i(\lambda_i\rho_0)=C_j(\lambda_j\rho_0)$ for all $i,j=1,\ldots,n$ is guaranteed by the monotonicity conditions as \eqref{eq:monotonicity}, replacing $n=3$ by the generic $n$. As seen in Example \ref{example2uni} (see Figure \ref{figure5}), in the case of more complex networks, the conditions are much more involved and less treatable, because of the peculiar characteristics of the problem. The topology of the network in fact strongly affects the monotonicity property, the way of representing it and, ultimately, its applicability. However, we point out that if all the single costs $C_i$ are strictly monotone, then they will certainly satisfy the corresponding monotonicity property.
\end{remark}
\begin{remark}
In the two examples here presented, the switching instants are a-priori fixed for all agents, and hence they do not enter in the optimization process performed by the single agent. In our model, in the previous sections, we instead consider also the switching time as well as the decisional time as part of the control for the agents, and the costs also depend on them. This fact obviously makes the situation much more complicated in order to establish a reasonable condition for the uniqueness of the mean-field game.
\end{remark}
\begin{remark}
The monotonicity conditions \eqref{eq:monotonicity} and \eqref{eq:meglio} and their possible generalization to more complicated networks, only guarantee the  uniqueness of the possible $n$-string of convex coefficients but not, in general, its existence. Note that, if the (unique) solution presents some $\lambda_i=0$, then it means that the corresponding node will be not reached by the equilibrium, but anyway, even with zero mass, that node produces the same cost as the others. Moreover, we may not have existence of the $n$-string convex solution. Looking at Example \ref{example1uni} (generalized to $n$ intermediate nodes), this means that we do not have a $n$-string which gives the fraction of mass switching to the $n$ nodes. This means that there is at least a node which must be not considered in the game from the beginning. For example, a node $p_i$ such that $C_i(\lambda_i\rho_0)>C_j(\lambda_j\rho_0)$ for all $j\neq i$ and $\lambda_i,\lambda_j$: it is a too expensive node, no one will switch to it. In this situation, the actual game is with just $n-1$ nodes and not with $n$. Hence, one must look for a possible unique $(n-1)$-string of convex combination solving the corresponding problem without that node. Proceeding in this way, one can find a possible unique $m$-string, and will set the other components to $0$: no flow through such nodes. However note that in our model, in the previous sections, we have also the time spent on the node at our disposal, which possibly modulate the paid cost, and hence the situation is more flexible but less prone to have a good condition for uniqueness.
\end{remark}
The points and the questions of these last remarks are certainly worth investigating and may be the argument of future studies.
\section{On the convexity of $V$ and single-valued feature of $\phi$}
\label{appendixB}
Let us assume
$$
C(p,p',t,\tau,\rho)=\frac{\bar C(p,p',\rho)}{\tau-t},\quad\tau\longmapsto\tilde C(\bar p,\tau)\ \text{strictly decreasing}.
$$
In particular, $\bar C$ does not explicitly depend on $t$ and $\tau$, for example
$$
\bar C(p,p',\rho)=\frac{a(p)}{T}\int_0^T\rho_{p}(s)ds+\frac{a(p')}{T}\int_0^T\rho_{p'}(s)ds
$$
for some weight $p\longmapsto a(p)$. A possible strictly non-decreasing $\tilde C$ is $\tilde C(\bar p,\tau)=T-\tau$.

Let $p_{1}$ be a node directly linked to $\bar p$, i.e. $\sum_ip_1^i=N-1$, and let $t<T$. Hence we have
$$
V(p_{1},t)=\inf_{\tau\in]t,T]}\left\{\frac{\bar C(p_{1},\bar p,\rho)}{\tau-t}+\tilde C(\bar p,\tau)\right\}=\frac{\bar C(p_{1},\bar p,\rho)}{T-t}+\tilde C(\bar p,T).
$$
Therefore, $t\longmapsto V(p_{1},t)$ is strictly convex and $\varphi(t)=T$ is single-valued.

Now, let $p_{2}$ be a node linked to $\bar p$ with two switches, i.e. $\sum_ip_2^i=N-2$, and let $p_{1}\in{\cal I}_{p_{2}}$ and $t<T$. We consider the function
$$
\psi_{p_2,p_1}:]t,T[\ni\tau\longmapsto V(p_{1},\tau)+\frac{\bar C(p_2,p_1,\rho)}{\tau-t}=\frac{\bar C(p_{1},\bar p,\rho)}{T-\tau}+\tilde C(\bar p,T)+\frac{\bar C(p_2,p_1,\rho)}{\tau-t}.
$$
Note that $\lim_{\tau\to t^+}\psi_{p_2,p_1}(\tau)=\lim_{\tau\to T^-}\psi_{p_2,p_1}(\tau)=+\infty$. Hence, the minimization problem
$$
\inf_{\tau\in]t,T[}\psi_{p_2,p_1}(\tau)
$$
has a solution $\varphi_{p_2,p_1}(t)\in]t,T[$ and it must be
\begin{equation}
\label{eq:firstorder}
\frac{\bar C(p_1,\bar p,\rho)}{(T-\varphi_{p_2,p_1}(t))^2}-\frac{\bar C(p_2,p_1,\rho)}{(\varphi_{p_2,p_1}(t)-t)^2}=0,
\end{equation}
which gives a unique possible point of minimum
$$
\varphi_{p_2,p_1}(t)=\frac{\sqrt{\frac{\bar C(p_2,p_1,\rho)}{\bar C(p_1,\bar p,\rho)}}T+t}{\sqrt{\frac{\bar C(p_2,p_1,\rho)}{\bar C(p_1,\bar p,\rho)}}+1}\in]t,T[,
$$
and note that $\varphi$ is strictly increasing and linear and hence derivable. Moreover, its derivative satisfies
\begin{equation}
\label{eq:lessthan1}
0<\varphi'_{p_2,p_1}(t)<1.
\end{equation}
We now consider the function
$$
V_{p_2,p_1}:t\longmapsto\psi_{p_2,p_1}(\varphi_{p_2,p_1}(t))=\frac{\bar C(p_{1},\bar p,\rho)}{T-\varphi_{p_2,p_1}(t)}+\tilde C(\bar p,T)+\frac{\bar C(p_2,p_1,\rho)}{\varphi_{p_2,p_1}(t)-t},
$$
which represents the optimum when, being in $p_2$ at time $t$, the agent decides that it will switch to $p_1$ before $T$, that is it will perform the path $p_2\to p_1\to \bar p$. Such a function is then twice derivable and it is strictly convex in $]0,T[$. Indeed, taking account of (\ref{eq:firstorder}) and (\ref{eq:lessthan1}), it is
$$
V''_{p_2,p_1}(t)=\frac{2\bar C(p_2,p_1,\rho)(\varphi_{p_2,p_1}(t)-t)(1-\varphi'_{p_2,p_1}(t))}{(\varphi_{p_2,p_1}(t)-t)^4}>0.
$$
Note that we do not need the second derivative of $\varphi_{p_2,p_3}$ (even if it exists, in our example) because in the calculation of $V'_{p_2,p_1}$ it cancels in view of (\ref{eq:firstorder}). Finally, note that $\lim_{t\to T^-}V_{p_2,p_1}(t)=+\infty$.

Now, we take $p_3$ such that $p_2\in{\cal I}_{p_3}$ and consider the function
$$
V_{p_3,p_1,p_1}:t\longmapsto\inf_{\tau\in]t,T[}\left\{V_{p_2,p_1}(\tau)+\frac{\bar C(p_3,p_2,\rho)}{\tau-t}\right\},
$$
which represents the optimum when, being in $p_3$ at time $t$, the agent decides that it will perform the path $p_3\to p_2\to p_1\to\bar p$. Note that the function $]t,T[\ni\tau\longmapsto\psi_{p_3,p_2,p_1}$, inside the minimization, is twice derivable and satisfies $\lim_{\tau\to t^+}\psi_{p_3,p_2,p_1}(\tau)=\lim_{\tau\to T^+}\psi_{p_3,p_2.p_4}(\tau)=+\infty$. Hence the minimization process has a solution $\varphi_{p_3,p_2,p_1}(t)\in]t,T[$, and such a solution is unique. Indeed, again, it must be
\begin{equation}
\label{eq:firstorder2}
V'_{p_2,p_1}(\varphi_{p_3,p_2,p_1}(t))=\frac{\bar C(p_3,p_2,\rho)}{(\varphi_{p_3,p_2,p_1}(t)-t)^2}.
\end{equation}
Whereas $\tau\longmapsto V'_{p_2,p_1}(\tau)$ is strictly increasing (being $V_{p_2,p_1}$ strictly convex) and $\tau\longmapsto\bar C(p_3,p_2,\rho)/(\tau-t)^2$ is strictly decreasing, the solution $\varphi_{p_3,p_2,p_1}(t)\in]t,T[$ is unique. Moreover, by the Implicit Function Theorem, $\varphi_{p_3,p_2,p_1}$ is derivable. Differentiating the equality (\ref{eq:firstorder2}), we get (we write $\varphi$ for $\varphi_{p_3,p_2,p_1}$)
$$
\left(V''_{p_2,p_1}(\varphi(t)+\frac{2\bar C(p_3,p_1,\rho)(\varphi(t)-t)}{(\varphi(t)-t)^4}\right)\varphi'(t)=\frac{2\bar C(p_3,p_2,\rho)(\varphi(t)-t)}{(\varphi(t)-t)^4},
$$
from which, being $V''_{p_2,p_1}>0$ and $\varphi(t)>t$, we get
\begin{equation}
\label{eq:lessthan12}
0<\varphi'_{p_3,p_2,p_1}(t)<1
\end{equation}
and in particular $\varphi_{p_1,p_2,p_3}$ is strictly increasing. Now, we prove that (still denoting $\varphi_{p_3,p_2,p_1}$ by $\varphi$)
$$
t\longmapsto V_{p_3,p_2,p_1}(t)=V_{p_2,p_1}(\varphi(t))+\frac{\bar C(p_3,p_2,\rho)}{\varphi(t)-t}
$$
is strictly convex. Indeed, differentiating two times, taking account of (\ref{eq:firstorder2}) and (\ref{eq:lessthan12}), we get again
$$
V''_{p_3,p_2,p_1}(t)=\frac{2\bar C(p_3,p_2,\rho)(\varphi(t)-t)(1-\varphi'(t))}{(\varphi(t)-t)^4}>0.
$$
Again, note that we do not need the second derivative of $\varphi$ (even if it exists, in our example) because in the calculation of $V'_{p_3,p_2,p_1}$ it cancels in view of (\ref{eq:firstorder2}). Finally note that $\lim_{t\to T^-}V_{p_3,p_2,p_1}(t)=+\infty$.

Proceeding in this way we obtain that, for every path $p_n\to p_{n-1}\to\cdots\to p_1\to\bar p$, the function
$$
V_{p_n,p_{n-1},\ldots,p_1}(t)=\inf_{\tau\in]t,T[}\left\{V_{p_{n-1},\ldots,p_1}(t)+\frac{\bar C(p_n,p_{n-1},\rho)}{\tau-t}\right\}
$$
is realized by a unique $\tau=\varphi_{p_n,\ldots,p_1}(t)\in]t,T[$, it is strictly convex, and $\varphi_{p_n,\ldots,p_1}$ is strictly increasing with derivative less than $1$.

We finally obtain that the value function, for all $p\neq\bar p$ and  $t<T$,
$$
V(p,t)=\inf_{\substack{\tau\in]t,T]\\p'\in{\cal I}_p}}\left\{V(p',\tau)+\frac{\bar C(p,p',\rho)}{\tau-t}\right\},
$$
is realized by a unique, strictly increasing (for $t$ such that $\varphi(t)>T$) single-valued function $t\longmapsto\tau=\varphi(t)\in]t,T]$, giving the optimal instant $\tau\in]t,T]$ for switching to the optimal node $p'\in{\cal I}_p$.

\end{document}